\documentclass[11pt,oneside]{amsart}
\usepackage{bm,cancel,ulem}
\usepackage{amssymb,amsfonts}
\usepackage{amsmath,latexsym,enumerate}
\usepackage{mathrsfs}
\usepackage[top=1in, bottom=1in, left=1.25in, right=1.25in, marginparwidth=23mm, marginparsep=2mm]{geometry}

\usepackage[unicode,breaklinks=true,colorlinks=true]{hyperref}

\usepackage[svgnames]{xcolor}

\newtheorem{thm}{Theorem}[section]
\newtheorem{cor}[thm]{Corollary}
\newtheorem{lem}[thm]{Lemma}

\theoremstyle{definition}

\newtheorem{rem}[thm]{Remark}
\numberwithin{equation}{section}

\newcommand{\R}{\mathbb R}
\newcommand\C{\mathcal{C}}

\newcommand{\cS}{\mathcal{S}}
\newcommand{\cT}{\mathcal{T}}

\newcommand{\dy}{\,{\rm d}y}
\newcommand{\dx}{\,{\rm d}x}
\newcommand{\dw}{\,{\rm d}w}
\newcommand{\dz}{\,{\rm d}z}
\newcommand{\dt}{\,{\rm d}t}
\newcommand{\ds}{\,{\rm d}s}

\newcommand{\dtau}{\,{\rm d}\tau}
\newcommand{\dxi}{\,{\rm d}\xi}

\newcommand{\lec}{\lesssim}
\renewcommand{\div}{\mathop{\rm div}\nolimits}
\newcommand{\norm}[1]{\left\|#1\right\|}
\newcommand{\al}{\alpha}
\newcommand{\pd}{\partial}

\newcommand{\ep}{\varepsilon}
\newcommand{\de}{\delta}

\newcommand{\loc}{\mathrm{loc}}
\renewcommand\rightmark{The Poisson kernel of Stokes system in half space with Navier BC}

\begin{document}
\title{Poisson kernel and blow-up of the second derivatives near the boundary for Stokes equations with Navier boundary condition}%
\author[H. Chen]{Hui Chen}%
\address[H. Chen]
{School of Science, Zhejiang University of Science and Technology, Hangzhou, 310023, People's Republic of China }
\email{chenhui@zust.edu.cn}
\author[S. Liang]{Su Liang}
\address[S. Liang]
{Department of Mathematics, University of British Columbia, Vancouver, BC V6T1Z2, Canada }
\email{liangsu96@math.ubc.ca}
\author[T.-P. Tsai]{Tai-Peng Tsai}
\address[T.-P. Tsai]
{Department of Mathematics, University of British Columbia, Vancouver, BC V6T1Z2, Canada }
\email{ttsai@math.ubc.ca}

\date{}  

\begin{abstract}
We derive the explicit Poisson kernel of Stokes equations in the half space with nonhomogeneous Navier boundary condition (BC) for both infinite and finite slip length.  By using this kernel, for any $q>1$, we construct a finite energy solution of Stokes equations with Navier BC in the half space, with bounded velocity and velocity gradient, but having unbounded second derivatives in $L^q$ locally near the boundary. While the Caccioppoli type inequality of Stokes equations with Navier BC is true for the first derivatives of velocity, which is proved by us in [CPAA 2023], this example shows that the corresponding  inequality for the second derivatives of the velocity is not true. Moreover, we give an alternative proof of the blow-up using a shear flow example, which is simple and  is the solution of both Stokes and Navier--Stokes equations.
\end{abstract}

\maketitle


\vskip 0.2cm


\tableofcontents

\section{Introduction}

We consider the following Stokes system in $\R^{n}_+\times \R$, $n\geq2$,
\begin{equation}\label{Stokes_Eq}
\partial_{t}\bm{u}-\Delta\bm{u}+\nabla p=0,\ \ \  \div \bm{u}=0, \end{equation}
with nonhomogeneous Navier boundary condition (Navier BC) on $\partial\R^{n}_+\times \R$
\begin{equation}\label{Navier-BC}
\partial_n u_{k}-\alpha u_{k}=a_k,\ \ u_n=a_n,
\end{equation}
for $1\leq k\leq n-1$. The system can be also considered for $t \in (0,\infty)$ with an initial condition at $t=0$.
Here $\bm{u}(x,t)=(u_1,u_2,\ldots, u_n)$ is the velocity field, $p(x,t)$ is the pressure, $\alpha \geq 0$ is the friction coefficient, and $\bm{a}(x',t)=(a_1,a_2,\ldots,a_n)$ is the boundary value. For $\alpha=0$, \eqref{Navier-BC} is reduced to Lions boundary condition. When $\alpha>0$, its inverse $1/\alpha$ has the unit of length and is called the slip length. We refer to our previous paper \cite{Chen2023} and the reference therein for more detailed introduction of the physical meaning, and the historical study on the Stokes and Navier--Stokes equations under Navier BC. 

In this paper, we continue to study the local regularity of the solution of the Stokes equations \eqref{Stokes_Eq} near boundary. We recall some developments along this line. For the  Stokes system in the half space with Dirichlet BC, Seregin \cite[Lemma 1.1]{Seregin2009} showed spatial smoothing ($\nabla^2 \bm{u},\nabla p\in L^{q,r}(Q_{1/2}^+)$, $1<q,r<\infty$)  if one assumes $\bm{u}, \nabla \bm{u}, p \in L^{q,r}(Q_1^+)$, where $Q_R^+$ denotes parabolic cylinders (see Section \ref{Sec 2} for definition). However, without any assumption on the pressure, the smoothing in spatial variables fails due to non-local effect of the pressure. The first counter example constructed by K. Kang \cite{Kang2004a} shows that there exists a bounded weak solution of the Stokes equations  whose normal derivative is unbounded near boundary,
\begin{equation*}
\sup _{Q_{\frac{1}{2}}^{+}}\left|\partial_{n} \bm{u}\right|=\infty, \quad \sup _{Q_1^{+}}|u|+\|\nabla \bm{u}\|_{L^2\left(Q_1^{+}\right)}<\infty.
\end{equation*}
Seregin-\v Sver\'ak \cite{Seregin2010} found a simplified example in the form of a shear flow such that its
gradient is unbounded near boundary, although the velocity field is locally bounded. Recently, Chang-Kang \cite[Theorem 1.1]{Chang2023} proved that for any $1<q<\infty$, a bounded very weak solution (depending on $q$) can be constructed whose derivatives are unbounded in $L^{q}$,
\begin{equation}\label{CK}
\|\nabla \bm{u}\|_{L^q (Q_\frac{1}{2}^{+})}=\infty, \quad\|\bm{u}\|_{L^{\infty}\left(Q_1^{+}\right)}<\infty.
\end{equation}
See \cite{Chang2020,Kang2022,Kang2023,Seregin2000} for related study along this line  and \cite{Gustafson2006} for partial regularity results on the boundary with Dirichlet BC.

For Stokes system \eqref{Stokes_Eq} in the half space with  Navier BC, we  proved in \cite[Theorem 1.1]{Chen2023} the following Caccioppoli's inequality,
\begin{equation}\label{main1}
\| \nabla\bm{u}\|_{L^{q,r}(Q^{+}_\frac{1}{2})}\lesssim\|\bm{u}\|_{L^{q,r}\left(Q_1^+\right)},\ 1<q,r<\infty.
\end{equation}
It is a very important distinction in comparison to the Dirichlet BC where Caccioppoli's inequality fails near boundary by \eqref{CK}. 
Moreover, we also proved in \cite[Theorem 1.3]{Chen2023} that $\nabla^m\bm{u}\in L^{\infty}(Q^+_{\frac{1}{2}})$, $m\geq1$ for bounded velocity $\bm{u}$, if the friction coefficient $\alpha=0$,
\begin{equation}\label{main2}
\| \nabla^m\bm{u}\|_{L^{\infty}(Q^{+}_\frac{1}{2})}\lesssim\|\bm{u}\|_{L^{\infty}\left(Q_1^+\right)}.
\end{equation}
Similar study was initiated by Dong-Kim-Phan \cite{Dong2022}, where boundary second derivative estimates for generalized Stokes system with VMO coefficients under Lions BC were proved. The main motivation of this paper is to study the regularity criteria of higher order derivatives of $\bm{u}$ near the boundary when assuming non-zero friction coefficient. It will be proved that $\nabla\bm{u}$ can not be replaced by $\nabla^2\bm{u}$ in inequality \eqref{main1}.

At first, we derive the explicit Poisson kernel of Stokes system with Navier BC, which allows us to express $\bm{u}$ and $p$ in terms of $\bm{a}$. 

\begin{thm}[Poisson kernel]\label{thm_main} For $n\geq 2$, $\al\ge0$, $1\leq k,j\leq n$, let the kernels $P_{kj}=P_{kj}^{0}+\alpha P_{kj}^{\alpha}$ and $g_{j}=g_{j}^{0}+\alpha g_{j}^{\alpha}$ be given as in Lemma \ref{lem_p0} and Lemma \ref{lem_pa}.
\begin{enumerate}[\textup{(}a\textup{)}]
\item Suppose $\bm{a}(x',t)\in C_c^{\infty}(\R^{n-1}\times \R)^n$ and define $\bm{u}(x,t)$, $p(x,t)$ by
\begin{align}
u_k(x,t)=&\int_{\R}\int_{\R^{n-1}}P_{kj}(x-y^\prime,t-s)a_j(y^\prime,s)\dy^{\prime}\ds,\label{u_expression}\\
p(x,t)=&\int_{\R}\int_{\R^{n-1}}g_{j}(x-y^\prime,t-s)a_j(y^\prime,s)\dy^{\prime}\ds. \nonumber
\end{align}
Then $\bm{u}(x,t)$, $p(x,t)$ are smooth functions satisfying \eqref{Stokes_Eq}--\eqref{Navier-BC}. Note that repeated index means summation. 

\item Suppose $\bm{a}\in L^q_c$, $q>1$, and define $\bm{u}$ by \eqref{u_expression}. Then $\bm{u}$ is a very weak solution of Stokes equations \eqref{Stokes_Eq} in $\R^n_+\times\R $ with Navier boundary condition \eqref{Navier-BC} with data $\bm{a}$ on $\pd\R^n_+\times\R$, i.e., $\bm{u}\in L^1_{\loc}(\overline{\R^n_+}\times\R)$, and
\begin{align}\label{cor_weak form3}
\begin{split}
\int_{\R}\int_{\R^n_+}\bm{u}\cdot(\pd_t+\Delta)\bm{\Phi}\dx\dt&=\int_{\R}\int_{\R^{n-1}}\sum_{k=1}^{n-1}a_k\Phi_k-a_n\pd_n\Phi_n\dx'\dt,\\
\int_{\R^n_+}\bm{u}\cdot\nabla \Psi\dx&=-\int_{\R^{n-1}}a_n\Psi \dx',\quad (a.e.\ t)
\end{split}
\end{align}
for any $\Psi\in C_c^1(\overline{\R^n_+})$ and any divergence free vector $\bm{\Phi}\in C_c^{2,1}(\overline{\R^n_+}\times\R)$ satisfying
\[
\partial_n \Phi_{k}-\alpha \Phi_{k}=0,\ \ (1 \le k < n),\quad \Phi_n=0,\quad \text{when\ } x_n=0.
\]
\end{enumerate}
\end{thm}

By $L^q_c$ we mean $L^q$ functions with compact support.

We refer to \cite[Definition 2.2]{Chen2023} for the motivation of the definition of very weak solutions. Above we have added nonhomogeneous boundary data in the definition. Next we use the Poisson kernel to construct a finite energy solution of Stokes equations, which has bounded velocity and velocity gradient, while its second derivatives blow up in $L^q$ norm. 

\begin{thm}[Unbounded second derivative in $L^q$]\label{thm_main2}
Let $n\ge2$, $1<q\leq \infty$, $\alpha>0$. There exists a very  weak solution of Stokes equations \eqref{Stokes_Eq} in $Q^+_1$ with homogeneous \textup{(}i.e. $\bm{a}=0$\textup{)} Navier boundary condition \eqref{Navier-BC} on $B'_1$ such that 
\begin{align}\label{thm_main2_blow}
\|\bm{u}\|_{L^{\infty}(Q_1^+)}+\|\nabla \bm{u}\|_{L^{\infty}(Q_1^+)}<\infty, \ \ \ \|\nabla^2 \bm{u}\|_{L^{q}(Q^+_{\frac{1}{2}})}=\infty.
\end{align}
Moreover, it is the restriction of a global solution of \eqref{Stokes_Eq} in $\R^n_+ \times (0,2)$ with finite global energy
\begin{equation}\label{global energy}
\sup_{0<t<2}\int_{\R_+^n}|\bm{u}(x,t)|^2\dx+\int_0^2\int_{\R_+^n}|\nabla \bm{u}(x,t)|^2\dx\dt<\infty,
\end{equation}
and this global solution satisfies Navier boundary condition \eqref{Navier-BC} on $\pd\R^n_+ \times (0,2)$ with bounded and compactly supported boundary data $\bm{a}$.
\end{thm}

\begin{rem}

(i) This theorem under Navier BC
is parallel to \eqref{CK} under zero BC. It tells us that the similar inequality of \eqref{main1} with $\nabla \bm{u}$ replaced by $\nabla^2 \bm{u}$ is not true when $\al>0$. 
By \eqref{main2}, the blow-up of $\nabla^2\bm{u}$ does not happen if  $\alpha=0$.

(ii) One of the key ingredients in our proof is the inequalities \eqref{ee} of Lemma \ref{Besov est} for the norm in the Besov space $\dot{B}^{\frac 12-\frac 1{2q}}_{q,q}(\R)$.
The first inequality of \eqref{ee} is proved by Chang-Kang \cite{Chang2023} using interpolation theorems in anisotropic Sobolev spaces. We give a simple, alternative proof of full \eqref{ee} in this paper.

(iii) The blow-up of $\nabla^2\bm{u}$ will not happen if we assume regularity of the pressure. Actually in \cite[Theorem 1.2]{Chen2023}, we proved that for $1<q,r<\infty$,
\begin{equation}\label{1.7}
\|\nabla \pd_t\bm{u}\|_{L^{q,r}(Q^+_{\frac{1}{2}})}+\|\nabla^3\bm{u}\|_{L^{q,r}(Q^+_{\frac{1}{2}})}+\|\nabla^2 p\|_{L^{q,r}(Q^+_{\frac{1}{2}})}\lesssim \|\bm{u}\|_{L^{q,r}(Q^+_1)}+\|p\|_{L^{q,r}(Q^+_1)}.
\end{equation}

(iv) An  example can be constructed to make $\|\nabla^4\bm{u}\|_{L^{q}(Q^+_{\frac{1}{2}})}=\infty$ while $\|p\|_{L^{\infty}(Q^+_1)}$ and $\|\bm{u}\|_{L^{\infty}(Q^+_1)}$ are bounded. See Remark \ref{rem63}.

(v) If we further require the boundary data to be continuous, we can still construct a solution with unbounded second derivative, similar to \cite{Kang2023}. See Remark \ref{rem5.4}.
\end{rem}

Inspired by Seregin-\v Sver\'ak \cite{Seregin2010}, where they construct a shear flow in $\R^3_+$ with unbounded derivatives of velocity near the boundary under Dirichlet BC, we will give another example which has the same kind of blow-up as \eqref{thm_main2_blow}. The advantage of this example is that it is simple and is the solution of both Stokes and Navier--Stokes equations. It should be pointed out that it is a shear flow and has no spatial decay. In particular, it has infinite global energy. 

\begin{thm}[Shear flow example]\label{thm_main3}
Let $n\geq 2$, $1<q\leq\infty$, $\alpha>0$. There exists a very weak solution of both Stokes equations \eqref{Stokes_Eq} and Navier-Stokes equations in $\R^n_+\times (0,2)$, with homogeneous $($i.e.~$\bm{a}=0)$ Navier boundary condition \eqref{Navier-BC} on $\pd \R^n_+\times (0,2)$ and zero initial condition at $t=0$ such that \eqref{thm_main2_blow} is true.
\end{thm}
In fact, the same proof can be used to give a second example of \eqref{CK} under zero BC. See Remark \ref{rem62}.

The rest of this paper is organized as follows:  We introduce notations and some preliminary results in Section \ref{Sec 2}. We derive the expression of the Poisson kernel in Section \ref{Sec 3}.  We give the estimate of the kernel and prove Theorem \ref{thm_main} in Section \ref{Sec 4}. We introduce the example for Theorem \ref{thm_main2} in Section \ref{Sec 5}. We give a proof of Theorem \ref{thm_main3}  in Section \ref{Sec 6}.

\section{Notations and preliminaries}\label{Sec 2}
For $x=(x_{1},\cdots,x_{n})\in \R^{n}$, denote $x^{\prime}=(x_{1},\cdots,x_{n-1})$ as the horizontal variable. Let $B'_R(x')=\{z'\in \R^{n-1}: |z'-x'|<R\}$, $\C_R^+=B'_R(0)\times (0,R)$, and $Q^+_R=\C_R^+\times(1-R^2,1)$. Denote the heat kernel
\begin{equation*}
\Gamma(x,t)=\begin{cases}\frac{1}{(4\pi t)^{n/2}}e^{\frac{-|x|^2}{4t}}& \text { for } t>0\\0& \text { for } t\leq 0\end{cases}
\end{equation*}
and the fundamental solution of $-\Delta$
\begin{equation*}
E(x)= \begin{cases}\frac{1}{n(n-2)|B_1|} \frac{1}{|x|^{n-2}} & \text { for } n \geq 3 \\ -\frac{1}{2 \pi} \log |x| & \text { if } n=2 \end{cases}.
\end{equation*}
Here $|B_1|=\frac{\pi^{\frac{n}{2}}}{\Gamma(\frac{n}{2}+1)}$ denotes the volume of the unit ball $B_1$ in $\R^n$. Notice that for $n\geq2$,
\begin{equation}\label{ET}
\int_{0}^{+\infty}\partial_{i}\Gamma(x,t)\dt=\partial_{i}E(x).
\end{equation}
Here we add the derivative $\partial_{i}$ to avoid singularity at $t=+\infty$ when $n=2$. Let $\mathcal{H}(t)$ be the Heaviside function. Denote $f\lesssim g$ if there is a constant $C$ such that $f\leq C g$, and $C$ depends on some variables which are clear in the context. Let $f*g$ be the convolution in terms of $(x',t)$. We use the following definition of Fourier transform
\begin{equation*}
\mathcal{F}f(\xi)=(2\pi)^{-\frac{n}{2}}\int_{\R^n}e^{-ix\cdot\xi}\,f(x)\dx,
\end{equation*}
\begin{equation*}
\mathcal{F}_{x'}f(\xi^{\prime},x_{n})=(2\pi)^{-\frac{n-1}{2}}\int_{\R^{n-1}}e^{-ix^{\prime}\cdot\xi^{\prime}}\,f(x)\dx^{\prime},
\end{equation*}
and
\begin{equation}\label{Fourier}
\tilde{g}(\xi^{\prime},x_{n},s)\equiv\mathcal{F}_{x',t}\,g=(2\pi)^{-\frac{n}{2}}\int_{\R}\int_{\R^{n-1}}e^{-i(x^{\prime}\cdot\xi^{\prime}+st)}g(x,t)\dx^{\prime}\dt.
\end{equation}

The following several lemmas (Lemma \ref{lem21}--Lemma \ref{lem24}) will be used in calculating the inverse Fourier transform to get the Poisson kernel. They are not new and are widely used for the derivation of Green tensor of the Stokes equations \cite{Kang2023,Lin2023}. 
We give their proofs here for the convenience of the readers.

\begin{lem}\label{lem21}
Suppose $x_{n}>0$. For $n\ge 3$ we have
\begin{equation*}
\mathcal{F}_{x'}^{-1}\left(\frac{e^{-x_n|\xi'|}}{|\xi'|}\right)=2(2\pi)^{\frac{n-1}{2}} E(x).
\end{equation*}
For $n=2$, we have
\[\mathcal{F}_{x'}^{-1}\left(-e^{-x_2|\xi_1|}\right)=2(2\pi)^{\frac{1}{2}} \partial_2E(x), \ \ \ \mathcal{F}_{x'}^{-1}\left(i\xi_1\frac{e^{-x_2|\xi_1|}}{|\xi_1|}\right)=2(2\pi)^{\frac{1}{2}} \partial_1E(x).\]
\end{lem}
\begin{proof}
It is well known that the Fourier transform of the fundamental solution of $-\Delta$ is 
\[\mathcal{F}_x\Big(E(x)\Big)=\frac{1}{(2\pi)^{\frac{n}{2}}|\xi|^2},\quad n\geq3.
\]
Hence,
\begin{align}
\mathcal{F}_{x'}\Big(E(x)\Big)&=\mathcal{F}_{x_n}^{-1}\Big(\frac{1}{(2\pi)^{\frac{n}{2}}|\xi|^2}\Big)\notag\\
&=\frac{1}{\sqrt{2\pi}}\int_{\mathbb{R}}\frac{1}{(2\pi)^{\frac{n}{2}}|\xi|^2}e^{i\xi_nx_n}\dxi_n\notag\\
&=\frac{1}{(2\pi)^{\frac{n+1}{2}}}\int_{\mathbb{R}}\frac{1}{(\xi_n+i|\xi'|)(\xi_n-i|\xi'|)}e^{i\xi_nx_n}\dxi_n \label{2.1}.
\end{align}
We can use Residue Theorem to calculate the integral in \eqref{2.1}. Since $x_n>0$, $e^{i\xi_nx_n}$ is bounded when $\xi_n$ is in the upper half of the complex plane. There is only one singular point $i|\xi'|$ in the upper half of the complex plane, so
\begin{equation*}
\mathcal{F}_{x'}\Big(E(x)\Big)=2\pi i\frac{1}{(2\pi)^{\frac{n+1}{2}}} \cdot\frac{1}{i|\xi'|+i|\xi'|}e^{i(i|\xi'|)x_n}=\frac{1}{2(2\pi)^{\frac{n-1}{2}}} \cdot\frac{e^{-x_n|\xi'|}}{|\xi'|}.
\end{equation*}

For $n=2$, we calculate them directly:
\begin{align*}
&\ \ \ \  \mathcal{F}_{x_1}^{-1}\left(-e^{-x_2|\xi_1|}\right)=-(2\pi)^{-\frac{1}{2}}\int_{\R}e^{ix_{1}\cdot z}e^{-x_{2}|z|}\dz\\
&=-(2\pi)^{-\frac{1}{2}}\int_{0}^{\infty}e^{-x_{2}z} \left(e^{ix_{1}z}+e^{-ix_{1}z}\right)\dz=-(2\pi)^{-\frac{1}{2}}\frac{2x_{2}}{x_{1}^2+x_{2}^2}=2(2\pi)^{\frac{1}{2}}\partial_2E(x),
\end{align*}
\begin{align*}
&\ \ \ \  \mathcal{F}_{x_1}^{-1}\left(i\xi_1\frac{e^{-x_2|\xi_1|}}{|\xi_1|}\right)=(2\pi)^{-\frac{1}{2}}\int_{\R}e^{ix_{1}\cdot z}\frac{iz}{|z|}e^{-x_{2}|z|}\dz\\
&=(2\pi)^{-\frac{1}{2}}i\int_{0}^{\infty}e^{-x_{2}z} \left(e^{ix_{1}z}-e^{-ix_{1}z}\right)\dz=-(2\pi)^{-\frac{1}{2}}\frac{2x_{1}}{x_{1}^2+x_{2}^2}=2(2\pi)^{\frac{1}{2}}\partial_1E(x).\qedhere
\end{align*}
\end{proof}

\begin{rem}\label{rem22}
When $n=2$, the Fourier transform $\mathcal{F}_{x_1}\big(2(2\pi)^{\frac{1}{2}}E(x_1,x_2)\big)$ is not the function $\frac{e^{-x_2|\xi_1|}}{|\xi_1|}$, which is not locally integrable, but the distribution
\begin{equation}\label{Fourier-of-log}
T_{x_2}(\xi_1)=\frac{e^{-x_2|\xi_1|}-1}{|\xi_1|}  +  \left(\frac{1}{|\xi_1|} \right)_1 + 2\gamma\de(\xi_1) =\left[ \left(\frac{1}{|\xi_1|} \right)_1 + 2\gamma\de(\xi_1)\right] \cdot e^{-x_2|\xi_1|},
\end{equation}
where $\gamma$ is the Euler–Mascheroni constant, and the distribution $ \left(\frac{1}{|\xi_1|} \right)_1$ is defined by
\begin{align}\label{def-absx-inverse}
\langle \left(\frac{1}{|\xi_1|} \right)_1,f\rangle=\int_{|\xi_1|<1} \frac{f(\xi_1)-f(0)}{|\xi_1|} \,d\xi_1+\int_{|\xi_1|\geq 1}  \frac{f(\xi_1)}{|\xi_1|}\,d\xi_1.
\end{align}
See \cite[page 258]{Schwartz1966}.
There is a good discussion by Mark Viola
in Math Stack Exchange, see
\href{https://math.stackexchange.com/questions/1186803/fourier-transform-of-logx2a2}{this link}.
\hfill $\square$
\end{rem}

By differentiating the equation in Lemma \ref{lem21} with respect to $x_n$, we have
\begin{cor}
For $n\ge2$ and $x_{n}>0$, we have
\[
\mathcal{F}_{x'}^{-1}\left(-e^{-x_n|\xi'|}\right)=2(2\pi)^{\frac{n-1}{2}} \partial_nE(x),
\]
\[
\mathcal{F}_{x'}^{-1}\left(|\xi'|e^{-x_n|\xi'|}\right)=2(2\pi)^{\frac{n-1}{2}} \partial^2_nE(x),
\]
\[
\mathcal{F}_{x'}^{-1}\left(-|\xi'|^2e^{-x_n|\xi'|}\right)=2(2\pi)^{\frac{n-1}{2}} \partial^3_nE(x).
\]
\end{cor}
\begin{lem}\label{lem23}
For $n\geq 1$ and $x_n>0$, we have
\begin{equation*}
\mathcal{F}^{-1}_{x',t}\left(\frac{e^{-x_n \sqrt{|\xi'|^2+is}}}{\sqrt{|\xi'|^2+is}}\right)=2(2\pi)^{\frac{n}{2}}\Gamma(x,t).
\end{equation*}
Note that in our definition, $\mathrm{Re}(\sqrt{|\xi'|^2+is})>0$ no matter $s>0$ or $s<0$.

\end{lem}
\begin{proof}
It is well known that the Fourier transform of heat kernel is 
\[\mathcal{F}_{x,t}\Big(\Gamma(x,t)\Big)=\frac{1}{(2\pi)^{\frac{n+1}{2}}(|\xi|^2+is)}.
\]
Hence,
\begin{align}
\mathcal{F}_{x',t}\Big(\Gamma(x,t)\Big)&=\mathcal{F}_{x_n}^{-1}\Big(\frac{1}{(2\pi)^{\frac{n+1}{2}}(|\xi|^2+is)}\Big)\notag\\
&=\frac{1}{\sqrt{2\pi}}\int_{\mathbb{R}}\frac{1}{(2\pi)^{\frac{n+1}{2}}(|\xi|^2+is)}e^{i\xi_nx_n}\dxi_n\notag\\
&=\frac{1}{(2\pi)^{\frac{n+2}{2}}}\int_{\mathbb{R}}\frac{1}{(\xi_n+i\sqrt{|\xi'|^2+is})(\xi_n-i\sqrt{|\xi'|^2+is})}e^{i\xi_nx_n}\dxi_n \label{2.2}.
\end{align}
We can use Residue Theorem to calculate the integral in \eqref{2.2}. Since $x_n>0$, $e^{i\xi_nx_n}$ is bounded when $\xi_n$ is in the upper half of the complex plane. There is only one singular point $i\sqrt{|\xi'|^2+is}$ in the upper half of the complex plane, so
\begin{align*}
\mathcal{F}_{x',t}\Big(\Gamma(x,t)\Big)&=2\pi i\frac{1}{(2\pi)^{\frac{n+2}{2}}} \cdot\frac{1}{i\sqrt{|\xi'|^2+is}+i\sqrt{|\xi'|^2+is}}e^{i(i\sqrt{|\xi'|^2+is})x_n}\\
&=\frac{1}{2(2\pi)^{\frac{n}{2}}} \cdot\frac{e^{-x_n\sqrt{|\xi'|^2+is}}}{\sqrt{|\xi'|^2+is}}. \qedhere
\end{align*}
\end{proof}

By differentiating the equation in Lemma \ref{lem23} with respect to $x_n$, we have
\begin{cor}\label{lem24}
For $n\ge2$ and $x_{n}>0$, we have
\[
\mathcal{F}_{x',t}^{-1}\left(-e^{-x_n\sqrt{|\xi'|^2+is}}\right)=2(2\pi)^{\frac{n}{2}} \partial_n\Gamma(x,t),
\]
\[
\mathcal{F}_{x',t}^{-1}\left(\sqrt{|\xi'|^2+is}e^{-x_n\sqrt{|\xi'|^2+is}}\right)=2(2\pi)^{\frac{n}{2}} \partial^2_n\Gamma(x,t),
\]
\[
\mathcal{F}_{x',t}^{-1}\left(-(|\xi'|^2+is)e^{-x_n\sqrt{|\xi'|^2+is}}\right)=2(2\pi)^{\frac{n}{2}} \partial^3_n\Gamma(x,t).
\]
\end{cor}

The following lemma will be used in the pointwise estimate of the Poisson kernel. For $w,z\in\R$, we denote
\begin{equation}\label{Awz}
A(x^{\prime},w,z,t)=\int_{\R^{n-1}}E(x^{\prime}-y^{\prime},w)\Gamma(y^{\prime},z,t)\dy^{\prime}.
\end{equation} 
For $n=3$, it is defined in \cite{Solonnikov1964} for $w=0$ or $z=0$, whose derivative estimates are given in \cite[(62, 63)]{Solonnikov1964} for $n=3$ and \cite[(2.11), (2.12)]{Kang2022} for general $n\ge2$.
Now we give similar estimates for our generalized function $A$.

\begin{lem}\label{lemAA}
For $n\geq2$ and integers $m,l,j,k\geq0$ with $l+j+n\geq 3$,
\begin{equation}\label{AA}
|\partial_t^m \partial_{x^{\prime}}^{l}\partial_w^{j}\partial_z^{k} A(x^{\prime},w,z,t)|\lesssim \frac{1}{\left(|x^{\prime}|^2+w^2+z^2+t\right)^{\frac{l+j+n-2}{2}}\left(z^2+t\right)^{\frac{k+1}{2}+m}}e^{-\frac{z^2}{10t}}.
\end{equation}
\end{lem}

This lemma follows from $A(x^{\prime},w,z,t)=A(x^{\prime},w,0,t) e^{-\frac{z^2}{4t}}$, the estimate of $A(x^{\prime},w,0,t)$ in \cite[(2.11)]{Kang2022}, and the inequalities
\begin{equation}
\frac1{a^2+t} e^{-c z^2/t}  \lec \frac1{a^2+t}\cdot \frac {t}{z^2+t} \lec \frac 1{a^2+z^2+t}.
\end{equation}

The next two lemmas will be useful to generate the blow-up in the proof of Theorem \ref{thm_main2}. 

\begin{lem}\label{lem210}
For $|x'|> 0$, we have
\begin{equation}\label{2.7}
\mathrm{p.v.}\int_{\R^{n-1}} \frac{y_1}{|y'|^{n}} e^{-\frac{|x'-y'|^2}{4t}}\dy'= (4\pi t)^{\frac{n-1}{2}} \frac{x_1}{|x'|^{n}}+Err(x',t)\frac{t^{\frac{n}{2}}}{|x'|^{n}},
\end{equation}
where $|Err(x',t)|\lesssim 1$. The first term dominates the second when $\sqrt t \ll |x_1|$.
\end{lem}
\begin{proof}
The proof is similar to \cite[Lemma 3.2]{Chang2023}. We divide $\R^{n-1}$ to three disjoint sets $D_1$, $D_2$ and $D_3$, which are defined by 
\[D_1=\{y'\in \R^{n-1}: |x'-y'|\leq \frac{|x'|}{10}\}, \ \ \ D_2=\{y'\in \R^{n-1}: |y'|\leq \frac{|x'|}{10}\},\]
and $D_3=\R^{n-1}\setminus (D_1\cup D_2)$. We then split the integral into three terms as follows:
\[\text{p.v.}\int_{\R^{n-1}} \frac{y_1}{|y'|^{n}} e^{-\frac{|x'-y'|^2}{4t}}\dy'=\int_{D_1}\cdots+\text{p.v.}\int_{D_2}\cdots+\int_{D_3}\cdots\equiv G_1+G_2+G_3.\]
Since $\text{p.v.}\int_{D_2}\frac{y_1}{|y'|^n}\dy'=0$, we have 
\[G_2=\int_{D_2}\frac{y_1}{|y'|^n}\left(e^{-\frac{|x'-y'|^2}{4t}}-e^{-\frac{|x'|^2}{4t}}\right)\dy'.\] 
Thus, using $|\nabla e^{-\frac{|x'|^2}{4t}}|\lesssim \frac{|x'|}{t}e^{-\frac{|x'|^2}{4t}}$, we have 
\begin{align*}
|G_2|&\lesssim \frac{|x'|}{t}e^{-\frac{|x'|^2}{4t}}\int_{D_2}\frac{1}{|y'|^{n-2}}\dy'\lesssim \frac{|x'|^2}{t}e^{-\frac{|x'|^2}{4t}}\lesssim \frac{t^{\frac{n}{2}}}{|x'|^{n}}.
\end{align*}
Since $\int_{|y'|>a}e^{-|y'|^2}\dy'\leq e^{-\frac{1}{2}a^2}\int_{|y'|>a}e^{-\frac{1}{2}|y'|^2}\dy'\lesssim e^{-\frac{1}{2}a^2}$, $a>0$, we have
\begin{equation*}
|G_3|\lesssim \frac{1}{|x'|^{n-1}}\int_{|x'-y'|\geq \frac{|x'|}{10}}e^{-\frac{|x'-y'|^2}{4t}}\dy'\lesssim \frac{1}{|x'|^{n-1}}t^{\frac{n-1}{2}}e^{-\frac{|x'|^2}{800t}}\lesssim \frac{t^{\frac{n}{2}}}{|x'|^{n}}.
\end{equation*}
We decompose $G_1$ in the following way, 
\begin{align*}
G_1&=\int_{|y'|\leq \frac{|x'|}{10}}\frac{x_1-y_1}{|x'-y'|^{n}}e^{-\frac{|y'|^2}{4t}}\dy'
=\int_{|y'|\leq \frac{|x'|}{10}}\left(\frac{x_1-y_1}{|x'-y'|^{n}}-\frac{x_1}{|x'|^{n}}\right)e^{-\frac{|y'|^2}{4t}}\dy'\\
& \ \ \ +\int_{\R^{n-1}}\frac{x_1}{|x'|^{n}}e^{-\frac{|y'|^2}{4t}}\dy'-\int_{|y'|\geq \frac{|x'|}{10}}\frac{x_1}{|x'|^{n}}e^{-\frac{|y'|^2}{4t}}\dy'= G_{11}+G_{12}+G_{13}.
\end{align*}
It's not hard to get $G_{12}=(4\pi t)^{\frac{n-1}{2}}\frac{x_1}{|x'|^{n}}$. Similar to $G_3$, we have 
\[|G_{13}|\lesssim \frac{1}{|x'|^{n-1}}t^{\frac{n-1}{2}}e^{-\frac{|x'|^2}{800t}}\lesssim \frac{t^{\frac{n}{2}}}{|x'|^{n}}.\]
Since $|\nabla \frac{x_1}{|x'|^{n}}|\lesssim \frac{1}{|x'|^{n}}$, we have
\[|G_{11}|\lesssim \frac{1}{|x'|^{n}} \int_{|y'|\leq \frac{|x'|}{10}}|y'|e^{-\frac{|y'|^2}{4t}}\dy'
\lesssim \frac{1}{|x'|^{n}}t^{\frac{n}{2}}\int_{\R^{n-1}}|z'|e^{-|z'|^2}\dz'\lesssim \frac{t^{\frac{n}{2}}}{|x'|^{n}}.\]
Combining the above estimates, we arrive at \eqref{2.7}.
\end{proof}

The following lemma is a generalized version of \cite[(4.14)]{Chang2023} and plays a crucial role in the proof of Theorem \ref{thm_main2}. Our lemma establishes both directions of the inequalities, while \cite[(4.14)]{Chang2023} only establishes the first inequality of \eqref{ee}, though only the first inequality is used later. Also our proof is more elementary and direct, without usage of the trace theorem of anisotropic Sobolev spaces.
Recall that for function $u$ in homogeneous Besov space $\dot{B}^s_{q,r}$ (see \cite[Definition 2.15]{Bahouri2011}), its norm is defined by
\[\|u\|_{\dot{B}^s_{q,r}}=\bigg(\textstyle \sum_{j\in \mathbb{Z}}2^{rjs}\|\Delta_j u\|^r_{L^q}\bigg)^{\frac{1}{r}},
\]
where $\Delta_j$ is the Littlewood-Paley operator.

\begin{lem}\label{Besov est}
Let $f(x,t)=\int_{\R}\Gamma_1(x,t-s)g(s)\ds$, $x>0$, $t\in \R$, where $\Gamma_1$ is the 1D heat kernel. We have that for $1\leq q<\infty$,
\begin{equation}\label{ee}
\norm{g}_{ \dot{B}^{\frac 12-\frac 1{2q}}_{q,q}(\R)} \lesssim 
\norm{\partial_x^2 f}_{L^q(\R_+ \times \R)} \lesssim \norm{g}_{ \dot{B}^{\frac 12-\frac 1{2q}}_{q,q}(\R)}.
\end{equation}
\end{lem}
\begin{proof}
Notice that
\begin{equation}\label{e1}
\mathcal{F}_{t} f(x,s)=\frac{1}{2\sqrt{is}}e^{-x\sqrt{is}}\cdot\mathcal{F}_t g(s).
\end{equation}
For fixed $x$, we have
\begin{equation}\label{e1b}
\|\partial_x^2 f\|_{L^{q}(\R)}\lesssim  \Sigma_{j\in \mathbb{Z}}\ \|\Delta_j \partial_x^2 f\|_{L^{q}(\R)}\lesssim \Sigma_{j\in \mathbb{Z}}\ 2^{j/2}e^{- c  x2^{j/2}} \|\Delta_j  g\|_{L^{q}(\R)}.
\end{equation}
Here $c>0$ is a constant. The first inequality is due to Littlewood-Paley decomposition and Minkowski's inequality. The proof of the second inequality is similar to \cite[Lemma 2.4]{Bahouri2011}. Specifically, by definition and \eqref{e1},
\begin{align*}
\|\Delta_j \partial_x^2 f\|_{L^{q}(\R)}=\|\Big(\hat{\phi}(2^{-j}s)\frac{\sqrt{is}}{2}e^{-x\sqrt{is}}\hat{g}(s)\Big)^{\vee}\|_{L^{q}(\R)},
\end{align*}
where $\hat{\phi}(s)$ is supported in some annulus $0<c_1<|s|<c_2<\infty$ with $\sum_{j \in \mathbb{Z}}\hat{\phi}(2^{-j}s)=1$. Let $\hat{\psi}(s)$ be supported in some annulus with $\hat{\psi}(s)=1$ on the support of $\hat{\phi}(s)$. By Young's inequality and scaling technique, we can write the above equation as
\begin{align*}
\|\Delta_j \partial_x^2 f\|_{L^{q}(\R)}&=\|\Big(\hat{\psi}(2^{-j}s)\hat{\phi}(2^{-j}s)\frac{\sqrt{is}}{2}e^{-x\sqrt{is}}\hat{g}(s)\Big)^{\vee}\|_{L^{q}(\R)}\\
&\leq \|\Big(\hat{\psi}(2^{-j}s)\frac{\sqrt{is}}{2}e^{-x\sqrt{is}}\Big)^{\vee}\|_{L^{1}(\R)}\cdot \|\Big(\hat{\phi}(2^{-j}s)\hat{g}(s)\Big)^{\vee}\|_{L^{q}(\R)}\\
&\leq2^{j/2}\, \|\Big(\hat{\psi}(s)\frac{\sqrt{is}}{2}e^{-x2^{j/2}\sqrt{is}}\Big)^{\vee}\|_{L^{1}(\R)}\cdot \|\Delta_j  g\|_{L^{q}(\R)}.
\end{align*}
Hence, what remains to show is 
\begin{align}\label{Young's eq cut off}
\|\Big(\hat{\psi}(s)\frac{\sqrt{is}}{2}e^{-x\sqrt{is}}\Big)^{\vee}\|_{L^{1}(\R)}\lesssim e^{- c  x}.
\end{align}
We have 
\begin{align*}
\Big(\hat{\psi}(s)\frac{\sqrt{is}}{2}e^{-x\sqrt{is}}\Big)^{\vee}&=\frac{1}{\sqrt{2\pi}}\int_{-\infty}^{\infty}\hat{\psi}(s)\frac{\sqrt{is}}{2}e^{-x\sqrt{is}}e^{ist}\ds\\
&=\frac{1}{\sqrt{2\pi}(1+t^2)}\int_{-\infty}^{\infty}(1-\pd_s^2)(e^{ist})\hat{\psi}(s)\frac{\sqrt{is}}{2}e^{-x\sqrt{is}}\ds\\
&=\frac{1}{\sqrt{2\pi}(1+t^2)}\int_{-\infty}^{\infty}(1-\pd_s^2)\Big(\hat{\psi}(s)\frac{\sqrt{is}}{2}e^{-x\sqrt{is}}\Big)e^{ist}\ds.
\end{align*}
Note that $\hbox{Re}(\sqrt{is})=\frac{1}{\sqrt 2}\sqrt{|s|}$ no matter $s>0$ or $s<0$. By using the fact
\begin{align*}
\Big|(1-\pd_s^2)\Big(\hat{\psi}(s)\frac{\sqrt{is}}{2}e^{-x\sqrt{is}}\Big)\Big|\lesssim e^{-cx},
\end{align*}
we arrive at \eqref{Young's eq cut off} and hence \eqref{e1b}. Therefore,
\begin{align*}
\int_{0}^{+\infty}\|\partial_x^2 f\|_{L^{q}(\R)}^q\dx&\lesssim \int_{0}^{+\infty}\left(\Sigma_{j\in \mathbb{Z}}\ 2^{j/2}e^{-cx 2^{j/2}} \|\Delta_j  g\|_{L^{q}(\R)}\right)^q \dx\\
&\lesssim \Sigma_{j\in \mathbb{Z}}\int_{0}^{+\infty} x^{-1+1/q} \left(2^{j/2}e^{-cx 2^{j/2}} \right)^{q-1+1/q} \dx\cdot\|\Delta_j  g\|_{L^{q}(\R)}^q\\
&\lesssim \Sigma_{j\in \mathbb{Z}}\ 2^{\frac{q-1}{2}j}\ \|\Delta_j  g\|_{L^{q}(\R)}^q 
\approx \|g\|_{\dot{B}^{\frac 12-\frac 1{2q}}_{q,q}}^q.
\end{align*}
Note that the second inequality is by H\"{o}lder's inequality $(\sum a_jb_j)^q\leq (\sum a_j^q)(\sum b_j^{q/(q-1)})^{q-1}$, where 
\[a_j=\left(2^{j/2}e^{-cx 2^{j/2}} \right)^{(q-1+1/q)/q}\|\Delta_j  g\|_{L^{q}(\R)},
\quad b_j=\left(2^{j/2}e^{-cx 2^{j/2}} \right)^{(q-1)/q^2},
\]
and $\sum b_j^{q/(q-1)} \lesssim x^{-1/q}$ since $\sup_{x>0} \sum_{j\in \mathbb{Z}}(x2^{j/2}e^{-cx 2^{j/2}})^{1/q}\lesssim 1$, see \cite[Lemma 2.35]{Bahouri2011}. Hence, the second inequality in \eqref{ee} is valid. 

To prove the first inequality in \eqref{ee}, we use the following identity from \eqref{e1},
\begin{equation}\label{e2}
\mathcal{F}_t g(s)=4\int_{0}^{+\infty} (2x\sqrt{is})e^{-x\sqrt{is}}\cdot\partial_x^2\mathcal{F}_t f(x,s)\dx.
\end{equation}
Therefore, by Minkowski and Young's inequalities and \eqref{Young's eq cut off} (and the inequality above it), we have
\begin{align*}
\|\Delta_j g\|_{L^{q}(\R)}&=\|\Big(\hat{\phi}(2^{-j}s)\hat{g}(s)\Big)^{\vee}\|_{L^{q}(\R)}\\
&=\|4\int_0^{\infty}\Big(\hat{\phi}(2^{-j}s)(2x\sqrt{is})e^{-x\sqrt{is}}\cdot(\partial_x^2\hat{f})\Big)^{\vee}\dx\|_{L^{q}(\R)}\\
&\lesssim\int_0^{\infty}x\norm{\big(\hat{\phi}(2^{-j}s)\sqrt{is}e^{-x\sqrt{is}}\big)^\vee }_{L^{1}(\R)}\cdot\|\partial_x^2f\|_{L^{q}(\R)}\dx \\
&\lesssim \int_{0}^{+\infty} x2^{j/2}\, e^{-cx2^{j/2}}\|\partial_x^2f\|_{L^{q}(\R)}\dx.
\end{align*}
Thus,
\begin{align*}
\Sigma_{j\in \mathbb{Z}}\ 2^{\frac{q-1}{2}j}\ \|\Delta_j  g\|_{L^{q}(\R)}^q &\lesssim \Sigma_{j\in \mathbb{Z}}\ 2^{\frac{q-1}{2}j}\ \left(\int_{0}^{+\infty} x2^{j/2}\,e^{-cx2^{j/2}}\|\partial_x^2f\|_{L^{q}(\R)}\dx\right)^q\\
&\lesssim \int_{0}^{+\infty}\left(\Sigma_{j\in \mathbb{Z}}\ x2^{j/2}\,e^{-cx2^{j/2}}\right)\cdot   \|\partial_x^2f\|_{L^{q}(\R)}^q\dx\\
&\lesssim \int_{0}^{+\infty}   \|\partial_x^2f\|_{L^{q}(\R)}^q\dx.
\end{align*}
For the second inequality we have used H\"{o}lder inequality 
\[(\int a(x)b(x)\dx)^q\leq (\int a(x)^q\dx)(\int b(x)^{q/(q-1)}\dx)^{q-1}, \] 
where 
\[a=\big(x2^{j/2}\,e^{-cx2^{j/2}}\big)^{1/q}\|\partial_x^2f\|_{L^{q}(\R)},\quad
b=\big(x2^{j/2}\,e^{-cx2^{j/2}}\big)^{1-1/q}.
\]
For the third inequality we used $\sup_{x>0} \sum_{j\in \mathbb{Z}}x2^{j/2}e^{-cx 2^{j/2}}\lesssim 1$ again.
Hence, the first inequality in \eqref{ee} is proved. 
\end{proof}

\section{Derivation of Poisson kernel}\label{Sec 3}
In this section we  derive the formula of Poisson kernel of system \eqref{Stokes_Eq}--\eqref{Navier-BC}. We assume $\bm{a}(x',t)\in C_c^{\infty}(\R^{n-1}\times\R)$, $\bm{u}(x,t)$ and $ p(x,t)$ ($\nabla p$ if $n=2$) are smooth functions satisfying \eqref{Stokes_Eq}--\eqref{Navier-BC} and vanishing sufficiently fast near infinity. We will use Fourier transform to calculate the Poisson kernel, which allows us to express $\bm{u}$ and $p$ in terms of $\bm{a}$. 

When $n\geq 3$, it is reasonable to assume $\bm{u}(x,t)$, $ p(x,t)$ vanish sufficiently fast near infinity, which makes the Fourier transforms of $\bm{u}$ and $p$ to be functions instead of distributions. However, we cannot expect pressure to vanish at infinity when $n=2$, since $\Delta p=0$ and the fundamental solution $E(x)=-\frac{1}{2 \pi} \log |x|$ for $n=2$ has no spatial decay (see also the first term in the RHS of \eqref{gn0}). Hence, the $n=2$ case is a bit subtle, and we need to treat it separately. Notice that $\pd_{1}\bm{u}$, $\pd_{1}p$ satisfy the Stokes equations \eqref{Stokes_Eq} and Navier BC \eqref{Navier-BC} with boundary value $\pd_{1}\bm{a}$. We can then  express $\pd_{1}\bm{u}$ and $\pd_{1}p$ in terms of $\pd_{1}\bm{a}$ for $n=2$. Hence, we can get the expression of $\bm{u}$ and $p$ in terms of $\bm{a}$ by the fact that $\bm{u}$ and $\nabla p$ vanish at infinity. The formula of pressure $p$ is subject to change up to a function $f(t)$. In the end we will get the same form of the Poisson kernel for both $n\geq 3$ and $n=2$.

We show the derivation of the Poisson kernel here for $n\geq3$ and $n=2$ is analogous using the above observation. We imitate Solonnikov’s treatment of the velocity for Stokes system in \cite{Solonnikov1977}, and decompose $\bm{u}$ and $p$ into two parts
\begin{equation}\label{expresion of u}
\bm{u}=\bm{v}+\nabla\varphi,\quad p=-\partial_{t}\varphi.
\end{equation}
Then $\bm{v}$ is solenoidal and solves the heat equation 
\begin{equation*}
\partial_{t}\bm{v}-\Delta \bm{v}=0,\quad    \nabla\cdot\bm{v}=0,
\end{equation*}
while $\varphi$ is harmonic
\[\Delta \varphi=0.
\]
Taking Fourier transform $\mathcal{F}_{x',t}$ on the above heat equation and Laplace equation, we have
\begin{align*}
\left(is+|\xi^{\prime}|^2\right)\tilde{\bm{v}}-\partial_{n}^2\tilde{\bm{v}}&=0,\\
|\xi^{\prime}|^2 \tilde{\varphi}-\partial_{n}^2\tilde{\varphi}&=0,
\end{align*}
where $\tilde{\bm{v}},\tilde{\varphi}$ are defined in \eqref{Fourier}.  Solving the above ODEs in $x_n>0$ and removing the exponentially growing part, we obtain
\begin{align*}
\tilde{\bm{v}}(\xi',x_n,s)&=\bm{\phi}(\xi',s)\cdot e^{-x_{n}\sqrt{is+|\xi^{\prime}|^2}},\\
\tilde{\varphi}(\xi',x_n,s)&=\psi(\xi',s)\cdot e^{-x_{n}|\xi^{\prime}|},
\end{align*}
for some function $\bm{\phi}=\left(\phi_1,\cdots,\phi_{n}\right)$ and $\psi$.
(Note $\bm{\phi}$ and $\varphi$ are different and are notations of \cite{Solonnikov1977}.)
By \eqref{expresion of u},
\begin{align}\label{eq3.2}
\tilde{\bm{u}}&=\bm{\phi}(\xi',s)\cdot e^{-x_{n}\sqrt{is+|\xi^{\prime}|^2}}+\left(i\xi_1,\cdots,i\xi_{n-1},-|\xi^{\prime}|\right)\psi(\xi',s)\cdot e^{-x_{n}|\xi^{\prime}|},\\
\tilde{p}&=-is\psi(\xi',s)\cdot e^{-x_{n}|\xi^{\prime}|}.\label{FT of pressure}
\end{align}
Next, taking Fourier transform of $\nabla\cdot\bm{v}=0$ and the Navier boundary condition \eqref{Navier-BC}, we get that 
\begin{equation*}
\sum_{h=1}^{n-1}i\xi_h\, \phi_h(\xi',s)-\sqrt{is+|\xi^{\prime}|^2}\,\phi_n(\xi',s)=0 ,
\end{equation*}
and for $1\leq k\leq n-1$,
\begin{align*}
-\left(\sqrt{is+|\xi^{\prime}|^2}+\alpha\right)\phi_k(\xi',s)-i\xi_{k}\left(|\xi^{\prime}|+\alpha\right)\psi(\xi',s)&=\tilde{a}_k(\xi',s), \\
\phi_n(\xi',s)-|\xi'|\psi(\xi',s)&=\tilde{a}_n(\xi',s).
\end{align*}
We solve the above system and get
\begin{align}
\psi(\xi',s)&=\frac{\sum_{h=1}^{n-1}i\xi_{h}\tilde{a}_{h}+\left(is+|\xi^{\prime}|^2+\alpha\sqrt{is+|\xi^{\prime}|^2}\right) \tilde{a}_{n}}{-is|\xi^{\prime}|+\alpha|\xi^{\prime}|\left(|\xi^{\prime}|-\sqrt{is+|\xi^{\prime}|^2}\right)},\label{psi}\\
\phi_k(\xi',s)&=-\frac{\tilde{a}_{k}}{\sqrt{is+|\xi^{\prime}|^2}+\alpha}-\frac{i\xi_{k}\left(|\xi^{\prime}|+\alpha\right)}{\sqrt{is+|\xi^{\prime}|^2}+\alpha}\psi(\xi',s),\notag\\
\phi_n(\xi',s)&=\tilde{a}_{n}+|\xi^{\prime}|\psi(\xi',s).\notag
\end{align}
Substituting the above expression back to \eqref{eq3.2}, we get that for $1\leq k\leq n-1$,
\begin{align}\label{FT of u}
\begin{split}
\tilde{u}_{k}(\xi',x_n,s)&=-\frac{\tilde{a}_{k}}{\sqrt{is+|\xi^{\prime}|^2}+\alpha} e^{-x_{n}\sqrt{is+|\xi^{\prime}|^2}}\\
&\ \ \ -i\xi_{k}\psi(\xi',s)\left(\frac{|\xi^{\prime}|+\alpha}{\sqrt{is+|\xi^{\prime}|^2}+\alpha}e^{-x_{n}\sqrt{is+|\xi^{\prime}|^2}}-e^{-x_{n}|\xi^\prime|}\right),\\
\tilde{u}_{n}(\xi',x_n,s)&=\tilde{a}_{n}e^{-x_{n}\sqrt{is+|\xi^{\prime}|^2}}+|\xi^\prime|\psi(\xi',s)\left(e^{-x_{n}\sqrt{is+|\xi^{\prime}|^2}}-e^{-x_{n}|\xi^{\prime}|}\right).
\end{split}
\end{align}
Taking inverse Fourier transform $\mathcal{F}_{x',t}^{-1}$ of equations \eqref{FT of pressure} and \eqref{FT of u}, we get $u_{k}=P_{kj}*a_{j}$ and $p=g_{j}*a_{j}$ for $1\leq k,j \leq n$. Observe that $\psi(\xi',s)$ can be split into two parts, the part with $\alpha=0$ and the remaining one
\begin{equation}\label{split psi}
\psi=\psi^0+\psi^\alpha,
\end{equation}
where
\begin{align}\label{psi0}
\psi^0=\frac{\sum_{h=1}^{n-1}i\xi_{h}\tilde{a}_{h}+\left(is+|\xi^{\prime}|^2\right) \tilde{a}_{n}}{-is|\xi'|}, \ \ \ 
\psi^\alpha=\alpha\frac{-\sum_{h=1}^{n-1}i\xi_{h}\tilde{a}_{h}+\sqrt{is+|\xi^{\prime}|^2}|\xi^{\prime}| \tilde{a}_{n}}{-is|\xi'|\left(|\xi^\prime|+\sqrt{is+|\xi^\prime|^2}+\alpha\right)}.
\end{align}
Hence, it seems
natural to split the velocity $\bm{u}$ and pressure $p$ into two parts, where the first part is the same expression with $\alpha=0$, and the second part is the difference. 
Namely,
\begin{align}\label{split of u}
\bm{u}=\bm{u}^0+\bm{u}^{\alpha},\quad p=p^0+p^{\alpha},
\end{align}
where for $1\leq k \leq n-1$,
\begin{align}\label{FT of u0}
\begin{split}
\tilde{u}^0_{k}(\xi',x_n,s)&=-\tilde{a}_{k}\frac{e^{-x_{n}\sqrt{is+|\xi^{\prime}|^2}}}{\sqrt{is+|\xi^{\prime}|^2}}  -i\xi_{k}|\xi'|\psi^0(\xi',s)\cdot\left(\frac{e^{-x_{n}\sqrt{is+|\xi^{\prime}|^2}}}{\sqrt{is+|\xi^{\prime}|^2}}-\frac{e^{-x_{n}|\xi^\prime|}}{|\xi'|}\right),\\
\tilde{u}^0_{n}(\xi',x_n,s)&=\tilde{a}_{n}e^{-x_{n}\sqrt{is+|\xi^{\prime}|^2}}+|\xi'|\psi^0(\xi',s)\cdot\left(e^{-x_{n}\sqrt{is+|\xi^{\prime}|^2}}-e^{-x_{n}|\xi^{\prime}|}\right),
\end{split}
\end{align}
and
\begin{equation}\label{FT of p0}
\tilde{p}^0(\xi',x_n,s)=-is\psi^0(\xi',s)e^{-x_{n}|\xi^{\prime}|}.
\end{equation}
In the following Lemma \ref{lem_p0}, we give the expression of the Poisson kernel of $\alpha=0$ case in physical variables $(x',x_n,t)$ by taking inverse Fourier transform of equations \eqref{FT of u0} and \eqref{FT of p0}.
Derivation of $\tilde{p}^{\alpha}$, $\tilde{\bm{u}}^{\alpha}$ will be  given before Lemma \ref{lem_pa}.

\begin{lem}\label{lem_p0}
Taking inverse Fourier transform of equations \eqref{FT of u0} and \eqref{FT of p0}, we get
\[u^0_{k}=P^0_{kj}*a_{j}, \ \ \ p^0=g^0_{j}*a_{j},\] 
for $1\leq k,j \leq n$. The followings are the expressions of $P^0_{kj}$ and $g^0_{j}$.
For $1\leq k \leq n$ and $1\leq j\leq n-1$,
\begin{equation}\label{3.10}
P_{kj}^{0}(x,t)=-2\delta_{kj}\Gamma(x,t)-2\mathcal{H}(t)\partial_{k}\partial_{j}\int_{t}^{+\infty}\Gamma(x,s)\ds,
\end{equation}
and for $1\leq k \leq n$ and $j=n$,
\begin{equation}\label{3.11}
P_{kn}^{0}(x,t)=-2\delta_{kn}\partial_{n}\Gamma(x,t)-2\mathcal{H}(t)\partial_{k}\partial_{n}^2\int_{t}^{+\infty}\Gamma(x,s)\ds-2\delta(t)\partial_{k}E(x),
\end{equation}
where $\mathcal{H}$ is the Heaviside function.
For $1\leq j \leq n-1$,
\begin{equation*}
g_{j}^{0}(x,t)=2\delta(t)\partial_{j}E(x),
\end{equation*}
and
\begin{equation}\label{gn0}
g_{n}^{0}(x,t)=2\delta^{\prime}(t)E(x)+2\delta(t)\partial_{n}^2E(x).
\end{equation}
\end{lem}

\begin{proof}
The main tools are Lemma \ref{lem21} to Lemma \ref{lem24}.
From \eqref{FT of u0} and \eqref{psi0}, for $1\leq k,j\leq n-1$, we have
\begin{align*}
P_{kj}^{0}(x,t)=&\,(2\pi)^{-n/2}\delta_{kj}\,\mathcal{F}_{x',t}^{-1}\left(-\frac{e^{-x_{n}\sqrt{is+|\xi^{\prime}|^2}}}{\sqrt{is+|\xi^{\prime}|^2}} \right)\\
&+(2\pi)^{-n/2}\partial_{k}\partial_{j}\mathcal{F}_{x',t}^{-1}\left(\frac{1}{is}\left(\frac{e^{-x_{n}\sqrt{is+|\xi^{\prime}|^2}}}{\sqrt{is+|\xi^{\prime}|^2}}-\frac{e^{-x_{n}|\xi^\prime|}}{|\xi^{\prime}|}\right)\right) \\
=&\,-2\delta_{kj}\,\Gamma(x,t)+2\partial_{k}\partial_{j}\left(\int_{0}^{t}\Gamma(x,s)\ds-\mathcal{H}(t)E(x)\right).
\end{align*}
Note that in the second step above we use the following equality 
\[\frac{1}{is}\left(\frac{e^{-x_{n}\sqrt{is+|\xi^{\prime}|^2}}}{\sqrt{is+|\xi^{\prime}|^2}}-\frac{e^{-x_{n}|\xi^\prime|}}{|\xi^{\prime}|}\right)=(\frac{1}{is}+\pi\delta(s))\left(\frac{e^{-x_{n}\sqrt{is+|\xi^{\prime}|^2}}}{\sqrt{is+|\xi^{\prime}|^2}}-\frac{e^{-x_{n}|\xi^\prime|}}{|\xi^{\prime}|}\right),\]
and then apply $\mathcal{F}^{-1}_{t}\left(\pi\delta(s)+\frac{1}{is}\right)=\sqrt{2\pi}\mathcal{H}(t)$. This trick will be used repeatedly in the later calculation. 

For $1\leq k\leq n-1$ and $j=n$, we have 
\begin{align}
P_{kn}^{0}(x,t)=&\,(2\pi)^{-n/2}\partial_{k}\mathcal{F}_{x',t}^{-1}\left(\frac{is+|\xi^{\prime}|^2}{is}\left(\frac{e^{-x_{n}\sqrt{is+|\xi^{\prime}|^2}}}{\sqrt{is+|\xi^{\prime}|^2}}-\frac{e^{-x_{n}|\xi^\prime|}}{|\xi^{\prime}|}\right)\right)\notag\\
=&\,(2\pi)^{-n/2}\partial_{k}\mathcal{F}_{x',t}^{-1}\left(\frac{1}{is}\left(\sqrt{is+|\xi^{\prime}|^2}e^{-x_{n}\sqrt{is+|\xi^{\prime}|^2}}-|\xi'|e^{-x_{n}|\xi^\prime|}\right)-\frac{e^{-x_{n}|\xi^\prime|}}{|\xi'|}\right) \notag\\
=&\,2\partial_{k}\partial_{n}^2\left(\int_{0}^{t}\Gamma(x,s)\ds-\mathcal{H}(t) E(x)\right)-2\delta(t)\partial_{k}E(x).\label{CZ2}
\end{align}

For $k=n$ and $1\leq j \leq n-1$, we have
\begin{align*}
P_{nj}^{0}(x,t)=&\,-(2\pi)^{-n/2}\partial_{j}\mathcal{F}_{x',t}^{-1} \left(\frac{1}{is}\left(e^{-x_{n}\sqrt{is+|\xi^{\prime}|^2}}-e^{-x_{n}|\xi^\prime|}\right)\right)\\
=&\,2\partial_{n}\partial_{j}\left(\int_{0}^{t}\Gamma(x,s)\ds-\mathcal{H}(t)E(x)\right),
\end{align*}
and for $k=j=n$, 
\begin{align*}
P_{nn}^{0}(x,t)&=(2\pi)^{-n/2}\,\mathcal{F}_{x',t}^{-1}\left( e^{-x_{n}\sqrt{is+|\xi^{\prime}|^2}}-\frac{is+|\xi^{\prime}|^2}{is}\left(e^{-x_{n}\sqrt{is+|\xi^{\prime}|^2}}-e^{-x_{n}|\xi^\prime|}\right)\right)\\
&=-2\partial_{n}\Gamma(x,t) +2\partial_{n}^3\left(\int_{0}^{t}\Gamma(x,s)\ds-\mathcal{H}(t) E(x)\right)-2\delta(t)\partial_{n}E(x).
\end{align*}
By \eqref{ET}, we obtain \eqref{3.10} and \eqref{3.11}. 

Analogously, from \eqref{FT of p0} and \eqref{psi0}, for $1\leq j \leq n-1$, we have
\begin{align*}
g_{j}^{0}(x,t)=(2\pi)^{-n/2}\partial_{j}\mathcal{F}_{x',t}^{-1}\left(\frac{1}{|\xi^{\prime}|}e^{-x_{n}|\xi^{\prime}|}\right)=2\delta(t)\partial_{j}E(x),
\end{align*}
and for $j=n$,
\begin{align*}
g_{n}^{0}(x,t)=(2\pi)^{-n/2}\mathcal{F}_{x',t}^{-1}\left(
\frac{is+|\xi^{\prime}|^2}{|\xi^{\prime}|}e^{-x_{n}|\xi^{\prime}|}\right)=2\delta^{\prime}(t)E(x)+2\delta(t)\partial_{n}^2E(x).
\end{align*}
The above shows the $g_j^0$ formulas in Lemma \ref{lem_p0}.
\end{proof}

Now we derive the expressions of $\tilde{\bm{u}}^{\alpha}$ and $\tilde{p}^{\alpha}$ defined in \eqref{split of u}. Denote $b=\sqrt{is+|\xi'|^2}$ for simplicity. 
Using the facts
\[
-\frac{1}{b+\alpha}=-\frac{1}{b}+\frac{\alpha}{b(b+\alpha)},\quad
\frac{|\xi'|+\alpha}{b+\alpha}=\frac{|\xi'|}{b}+\frac{\alpha(b-|\xi'|)}{b(b+\alpha)},
\]
and the splitting of $\psi(\xi',s)$ in \eqref{split psi}, we have
\begin{align}\label{FT of pa}
\tilde{p}^{\alpha}(\xi',x_n,s)
&=-is\psi^\alpha(\xi',s)e^{-x_{n}|\xi^{\prime}|}\notag\\
&=\frac{-\alpha\sum_{h=1}^{n-1}i\xi_{h}\tilde{a}_{h} +\alpha b|\xi'|\tilde{a}_{n}}{|\xi'|+b+\alpha} \cdot\frac{e^{-x_n|\xi'|}}{|\xi'|},
\end{align}
and for $1\leq k \leq n-1$,
\begin{align}\label{FT of uka}
\tilde{u}^{\alpha}_k(\xi',x_n,s)&=\frac{\alpha\tilde{a}_k}{b(b+\alpha)}e^{-x_n b}-i\xi_k \psi^0(\xi',s)\frac{\alpha(b-|\xi'|)}{b(b+\alpha)} e^{-x_n b}\notag\\
&\quad-i\xi_{k}\psi^\alpha(\xi',s)\left(\frac{|\xi'|+\alpha}{b+\alpha}e^{-x_{n}b}-e^{-x_{n}|\xi^\prime|}\right)\notag \\
&=\frac{\alpha\tilde{a}_k}{b(b+\alpha)}e^{-x_n b} +\alpha i\xi_k\left(\sum_{h=1}^{n-1}i\xi_{h}\tilde{a}_{h}\right)\cdot\Bigg(\frac{-1}{(|\xi'|+b+\alpha)is}\bigg(\frac{e^{-x_n b}}{b}-\frac{e^{-x_n|\xi'|}}{|\xi'|}\bigg) \notag\\
&\ \ \ +\frac{ e^{-x_n b}}{|\xi'|(|\xi'|+b+\alpha)b(b+\alpha)}\Bigg)\notag\\
&\ \ \ +\alpha i\xi_k\tilde{a}_{n}\Bigg(\frac{ b|\xi'|}{(|\xi'|+b+\alpha)is}\bigg(\frac{e^{-x_n b}}{b}-\frac{e^{-x_n|\xi'|}}{|\xi'|}\bigg) +\frac{ e^{-x_n b}}{|\xi'|(|\xi'|+b+\alpha)}\Bigg),
\end{align}
and
\begin{align}\label{FT of una}
\tilde{u}_{n}^\alpha(\xi',x_n,s)&=|\xi^\prime|\psi^\alpha(\xi',s)\left(e^{-x_{n}b}-e^{-x_{n}|\xi^{\prime}|}\right)\notag\\
&=\frac{\alpha\sum_{h=1}^{n-1}i\xi_{h}\tilde{a}_{h} -\alpha b|\xi'|\tilde{a}_{n}}{(|\xi'|+b+\alpha)is} \left(e^{-x_{n}b}-e^{-x_{n}|\xi'|}\right).
\end{align}

The following Lemma will be useful in later calculation.
\begin{lem}\label{lem_pa3}
\begin{align}
&\frac{1}{is(|\xi^\prime|+b+\alpha)}\left(\frac{e^{-x_n b}}{b}-\frac{e^{-x_n|\xi'|}}{|\xi'|}\right)\notag\\
&=-\frac{1}{2|\xi^{\prime}|b}\int_{0}^{+\infty}e^{-\alpha z}\int_{0}^{+\infty}e^{-(y_{n}+z)|\xi^{\prime}|}\left(e^{-|x_{n}+z-y_n|b}+e^{-(x_{n}+z+y_n)b}\right)\dy_{n}\dz \label{3.17}\\
&\quad+\frac{1}{2|\xi^{\prime}|b}\int_{0}^{+\infty}e^{-\alpha z}\int_{0}^{+\infty}e^{-(x_n+y_{n}+z)|\xi^{\prime}|}\left(e^{-|z-y_n|b}-e^{-(z+y_n)b}\right)\dy_{n}\dz,\notag
\end{align}
and
\begin{equation}\label{3.18}
\frac{1}{is(|\xi^\prime|+b+\alpha)}\left(e^{-x_n b}-e^{-x_n|\xi'|}\right)=-\frac{1}{is(|\xi^\prime|+b+\alpha)}\partial_{n}\left(\frac{e^{-x_n b}}{b}-\frac{e^{-x_n|\xi'|}}{|\xi'|}\right).
\end{equation}
\end{lem}
\begin{proof}
Notice that, for $w,z>0$,
\begin{align}
&\quad\int_0^{+\infty}e^{-(y_{n}+w)|\xi^{\prime}|}e^{-|z-y_n|b}\dy_n\notag\\
&=\int_0^{z}e^{-(y_{n}+w)|\xi^{\prime}|}e^{-(z-y_n)b}\dy_n+\int_{z}^{+\infty}e^{-(y_{n}+w)|\xi^{\prime}|}e^{-(y_n-z)b}\dy_n\notag\\
&=\frac{1}{b-|\xi'|}\left(e^{zb-z|\xi'|}-1 \right)e^{-w|\xi'|}e^{-zb}+\frac{1}{|\xi'|+b}e^{-z(|\xi'|+b)}e^{-w|\xi'|}e^{zb}\notag\\
&=-\frac{1}{b-|\xi'|}e^{-w|\xi'|}e^{-zb}+\left(\frac{1}{b-|\xi^{\prime}|}+\frac{1}{b+|\xi^{\prime}|}\right)e^{-(w+z)|\xi^{\prime}|},\label{3.21}
\end{align}
and
\begin{equation}\label{3.22}
\int_{0}^{+\infty}e^{-(y_{n}+w)|\xi^{\prime}|}e^{-(y_n+z)b}\dy_n=\frac{1}{b+|\xi'|}e^{-w|\xi'|}e^{-zb}.
\end{equation}
Combining \eqref{3.21} and \eqref{3.22}, we obtain
for $w,z>0$
\begin{align}
\frac{1}{is}e^{-w|\xi^{\prime}|}e^{-zb}&=-\frac{1}{2|\xi^{\prime}|}\int_{0}^{+\infty}e^{-(y_{n}+w)|\xi^{\prime}|}\left(e^{-|z-y_n|b}+e^{-(y_n+z)b}\right)\dy_{n}\notag\\
&\quad+\frac{b}{is|\xi^{\prime}|}e^{-(w+z)|\xi^{\prime}|},\label{3.19}
\end{align}
and
\begin{align}
\frac{1}{is}e^{-w|\xi^{\prime}|}e^{-zb}&=-\frac{1}{2b}\int_{0}^{+\infty}e^{-(y_{n}+w)|\xi^{\prime}|}\left(e^{-|z-y_n|b}-e^{-(y_n+z)b}\right)\dy_{n}\notag\\
&\quad+\frac{1}{is}e^{-(w+z)|\xi^{\prime}|}.\label{3.20}
\end{align}

Since
\begin{align*}
&\quad\frac{1}{is(|\xi^\prime|+b+\alpha)}\left(\frac{e^{-x_n b}}{b}-\frac{e^{-x_n|\xi'|}}{|\xi'|}\right)\\
&=\int_{0}^{+\infty}e^{-\alpha z}\frac{1}{is}\left(\frac{1}{b}e^{-z|\xi^{\prime}|}e^{-\left(x_{n}+z\right)b}-\frac{1}{|\xi^{\prime}|}e^{-(x_n+z)|\xi^{\prime}|}e^{-zb}\right)\dz,
\end{align*}
and by \eqref{3.19} and \eqref{3.20},
\begin{align*}
&\quad\frac{1}{is}\left(\frac{1}{b}e^{-z|\xi^{\prime}|}e^{-\left(x_{n}+z\right)b}-\frac{1}{|\xi^{\prime}|}e^{-(x_n+z)|\xi^{\prime}|}e^{-zb}\right)\\
&=-\frac{1}{2|\xi^{\prime}|b}\int_{0}^{+\infty}e^{-(y_{n}+z)|\xi^{\prime}|}\left(e^{-|x_{n}+z-y_n|b}+e^{-(x_{n}+z+y_n)b}\right)\dy_{n}\\
&\quad+\frac{1}{2|\xi^{\prime}|b}\int_{0}^{+\infty}e^{-(x_n+y_{n}+z)|\xi^{\prime}|}\left(e^{-|z-y_n|b}-e^{-(z+y_n)b}\right)\dy_{n},
\end{align*}
we arrive at \eqref{3.17}.
\end{proof}

\begin{lem}\label{lem_pa}
Taking inverse Fourier transform of equations \eqref{FT of pa}, \eqref{FT of uka}, \eqref{FT of una}, we get
\[u^{\alpha}_{k}=\alpha P^{\alpha}_{kj}*a_{j}, \ \ \  p^{\alpha}=\alpha g^{\alpha}_{j}*a_{j},\] 
for $1\leq k,j \leq n$. 
The following are the expressions of $P^{\alpha}_{kj}$ and $g^{\alpha}_{j}$.
For $1\leq k \leq n$ and $1\leq j\leq n-1$,
\begin{equation}\label{pkj1}
P_{kj}^{\alpha}=2 \delta_{kj} I_1+4 \delta_{k<n} \partial_k\partial_j I_2+2\partial_{k}\partial_{j} I_3-2\partial_{k}\partial_{j} I_4,
\end{equation}
with
\begin{equation}\label{def of I}
\begin{split}
I_1&=\int_{0}^{+\infty}e^{-\alpha z}\,\Gamma(x^{\prime},x_{n}+z,t)\dz,\\
I_2&=\int_{0}^{+\infty}\int_{0}^{+\infty}e^{-\alpha (w+z)}\int_{\R^{n-1}}E(x^{\prime}-y^{\prime},z)\Gamma(y^{\prime},x_{n}+w+z,t)\dy^{\prime}\dw\dz,\\
I_3&=\int_{0}^{+\infty}e^{-\alpha z}\int_{\R_+^{n}}E(x^{\prime}-y^{\prime},y_n+z)\Gamma(y^{\prime},x_{n}+z-y_n,t)\dy\dz\\
&\quad+\int_{0}^{+\infty}e^{-\alpha z}\int_{\R_+^{n}}E(x^{\prime}-y^{\prime},y_n+z)\Gamma(y^{\prime},x_{n}+z+y_n,t)\dy\dz,\\
I_4&=\int_{0}^{+\infty}e^{-\alpha z}\int_{\R_+^{n}}E(x^{\prime}-y^{\prime},x_{n}+y_n+z)\Gamma(y^{\prime},z-y_n,t)\dy\dz\\
&\quad-\int_{0}^{+\infty}e^{-\alpha z}\int_{\R_+^{n}}E(x^{\prime}-y^{\prime},x_{n}+y_n+z)\Gamma(y^{\prime},z+y_n,t)\dy\dz.
\end{split}
\end{equation}
For $1\leq k \leq n$ and $j=n$,
\begin{equation}\label{pkj2}
P_{kn}^{\alpha}= -4 \delta_{k<n}\partial_k\partial_n J_1-2 \partial_k\partial_n J_2+2\partial_k\partial_n J_3,
\end{equation}
with
\begin{equation}\label{def of J}
\begin{split}
J_1&=\int_{0}^{+\infty}e^{-\alpha z}\int_{\R^{n-1}}E(x^{\prime}-y^{\prime},z)\Gamma(y^{\prime},x_{n}+z,t)\dy^{\prime}\dz,\\
J_2&=\int_{0}^{+\infty}e^{-\alpha z}\int_{\R_+^{n}}\partial_{n}E(x^{\prime}-y^{\prime},y_n+z)\Gamma(y^{\prime},x_{n}+z-y_n,t)\dy\dz\\
&\quad+\int_{0}^{+\infty}e^{-\alpha z}\int_{\R_+^{n}}\partial_{n}E(x^{\prime}-y^{\prime},y_n+z)\Gamma(y^{\prime},x_{n}+z+y_n,t)\dy\dz,\\
J_3&=\int_{0}^{+\infty}e^{-\alpha z}\int_{\R_+^{n}}E(x^{\prime}-y^{\prime},x_{n}+y_n+z)\partial_{n}\Gamma(y^{\prime},z-y_n,t)\dy\dz\\
&\quad-\int_{0}^{+\infty}e^{-\alpha z}\int_{\R_+^{n}}E(x^{\prime}-y^{\prime},x_{n}+y_n+z)\partial_{n}\Gamma(y^{\prime},z+y_n,t)\dy\dz.
\end{split}
\end{equation}
For $1\leq j \leq n-1$,
\begin{equation*}
g_{j}^{\alpha}(x,t)=4\partial_j\int_{0}^{+\infty}e^{-\alpha z}\left(\int_{\R^{n-1}}E(x^{\prime}-y^{\prime},x_{n}+z)\partial_{n}\Gamma(y^{\prime},z,t)\dy^{\prime}\right)\dz.
\end{equation*}
For $j=n$,
\begin{equation*}
g_{n}^{\alpha}(x,t)=-4\partial_{n}\int_{0}^{+\infty}e^{-\alpha z}\left(\int_{\R^{n-1}} E(x^{\prime}-y^{\prime},x_{n}+z)\,\partial_{n}^2 \Gamma(y^{\prime},z,t)\dy^{\prime}\right)\dz.
\end{equation*} 
\end{lem}
\begin{proof}
The main tools are Lemma \ref{lem21} to Lemma \ref{lem24}
and Lemma \ref{lem_pa3}, and the following equation
\begin{equation}\label{alpha to int}
\frac{1}{w+\alpha}=\int_0^{+\infty}e^{-(w+\alpha)z}\dz, \quad (\mbox{Re}(w)>0).
\end{equation}

We first obtain the expression of $g_j^\alpha$ by inverting $\tilde g_j^\alpha$, the factor of $\tilde a_j$ in \eqref{FT of pa}. For $1\leq j \leq n-1$, we have
\begin{align*}
g_{j}^{\alpha}(x,t)&=-(2\pi)^{-n/2}\partial_{j}\mathcal{F}_{x',t}^{-1}\left(\frac{1}{|\xi^{\prime}|\left(|\xi^\prime|+b+\alpha\right)}e^{-x_{n}|\xi^{\prime}|}\right)\\
&=- (2\pi)^{-n/2}\partial_{j}\mathcal{F}_{x',t}^{-1}\left(\int_{0}^{+\infty}e^{-\alpha z}\cdot\frac{e^{-(x_n+z)|\xi'|}}{|\xi^{\prime}|}\cdot e^{-zb}\dz\right)\\
&=4\partial_j\left(\int_{0}^{+\infty}e^{-\alpha z}\int_{\R^{n-1}}E(x^{\prime}-y^{\prime},x_{n}+z)\partial_{n}\Gamma(y^{\prime},z,t)\dy^{\prime}\dz\right),
\end{align*}
and for $j=n$,
\begin{align*}
g_{n}^{\alpha}(x,t)&=(2\pi)^{-n/2}\mathcal{F}_{x',t}^{-1}\left(\frac{b}{|\xi^\prime|+b+\alpha}\,e^{-x_{n}|\xi^{\prime}|}\right)\\
&=(2\pi)^{-n/2}\mathcal{F}_{x',t}^{-1}\left(\int_0^{+\infty}e^{-\alpha z}\cdot e^{-(x_n+z)|\xi'|}\cdot b e^{-zb}\dz\right)\\
&=-4\int_{0}^{+\infty}e^{-\alpha z}\int_{\R^{n-1}}\partial_{n} E(x^{\prime}-y^{\prime},x_{n}+z)\,\partial_{n}^2 \Gamma(y^{\prime},z,t)\dy^{\prime}\dz.
\end{align*}

We next obtain the expression of $P_{kj}^\alpha$  by inverting $\tilde P_{kj}^\alpha$, the factor of $\tilde a_j$ in \eqref{FT of uka} and  \eqref{FT of una}. For $1\leq k\leq n $ and $1\leq j\leq n-1$, by \eqref{3.18} when $k=n$, we have
\begin{align*}
P_{kj}^{\alpha}(x,t)&=(2\pi)^{-n/2}\delta_{kj}\,\mathcal{F}_{x',t}^{-1}\left(\frac{1}{b\left(b+\alpha\right)} e^{-x_{n}b}\right)\\
&+ \delta_{k<n}(2\pi)^{-n/2}\partial_{k}\partial_{j}\mathcal{F}_{x',t}^{-1}\left(\frac{e^{-x_{n}b}}{|\xi^{\prime}|\left(b+\alpha\right)\left(|\xi^{\prime}|+b+\alpha\right)b}\right)\\
&-(2\pi)^{-n/2}\partial_{k}\partial_{j}\mathcal{F}_{x',t}^{-1}\left(\frac{1}{is\left(|\xi^{\prime}|+b+\alpha\right)}\left(\frac{e^{-x_{n}b}}{b}-\frac{e^{-x_{n}|\xi^{\prime}|}}{|\xi^{\prime}|}\right)\right).
\end{align*}
Using Lemmas \ref{lem21}, \ref{lem23} and equations \eqref{3.17} and \eqref{alpha to int}, we obtain \eqref{pkj1}.
For $1\leq k \leq n$  and $j=n$, by \eqref{3.18} when $k=n$, we have
\begin{align*}
P_{kn}^{\alpha}(x,t)&=\delta_{k<n}(2\pi)^{-n/2}\partial_{k}\mathcal{F}_{x',t}^{-1}\left(\frac{1}{|\xi^{\prime}|\left(|\xi^{\prime}|+b+\alpha\right)}e^{-x_{n}b}\right)\\
&+(2\pi)^{-n/2}\partial_{k}\mathcal{F}_{x',t}^{-1}\left(\frac{|\xi^{\prime}|b}{is\left(|\xi^{\prime}|+b+\alpha\right)}\left(\frac{e^{-x_{n}b}}{b}-\frac{e^{-x_{n}|\xi^{\prime}|}}{|\xi^{\prime}|}\right)\right).
\end{align*}
Using Lemma \ref{lem21}--Corollary \ref{lem24} and equations \eqref{3.17} and \eqref{alpha to int}, we obtain \eqref{pkj2}. 
\end{proof}

\begin{lem}\label{lem-boundary-convergence}
If $\bm{a}\in C_c^{1,0}(\R^{n-1}\times \R)$, and $\tilde{\bm{u}}(x,t)$ is defined by \eqref{FT of u}, then 
\begin{align}\label{cor_bc2}
\begin{split}
(\partial_n\tilde{u}_k-\alpha \tilde{u}_k)(x,t)\to \tilde{a}_k(x',t),\\
\tilde{u}_n(x,t)\to \tilde{a}_n(x',t),
\end{split}
\end{align}
in the $L^2$ norm as $x_n\to 0_+$.
\end{lem}
\begin{proof}
For $\bm{a}(x',t)\in C_c^{1,0}(\R^{n-1}\times \R)$, we have that $\tilde{\bm{a}}(\xi',s)\in L^2$ and $\tilde{\bm{a}}(\xi',s)|\xi'|\in L^2$. By \eqref{FT of u} we have for $1\leq k\leq n-1$,
\begin{align*}
(\pd_n\tilde{u}_k-\alpha \tilde{u}_k)(\xi',x_n,s)&=\tilde{a}_ke^{-x_n\sqrt{is+|\xi'|^2}}+i\xi_k(|\xi'|+\alpha)\left(e^{-x_{n}\sqrt{is+|\xi^{\prime}|^2}}-e^{-x_{n}|\xi^{\prime}|}\right)\psi,\\
\tilde{u}_{n}(\xi',x_n,s)&=\tilde{a}_{n}e^{-x_{n}\sqrt{is+|\xi^{\prime}|^2}}+|\xi^\prime|\left(e^{-x_{n}\sqrt{is+|\xi^{\prime}|^2}}-e^{-x_{n}|\xi^{\prime}|}\right)\psi,
\end{align*}
where $\psi$ is defined in \eqref{psi}.
By dominated convergence theorem, we have that $\tilde{a}_ke^{-x_n\sqrt{is+|\xi'|^2}}$ $\to\tilde{a}_k$ in $L^2$. Denote 
\[L=\frac{e^{-x_{n}\sqrt{is+|\xi^{\prime}|^2}}-e^{-x_{n}|\xi^{\prime}|}}{\sqrt{is+|\xi^{\prime}|^2}-|\xi'|}.\]
It's not hard to prove $|L|\leq x_ne^{-x_n|\xi'|}$, and hence $|\xi'|L$ is bounded. Notice that, using $is=b^2-|\xi'|^2$,
\begin{align*}
R_1:&=i\xi_k(|\xi'|+\alpha)\left(e^{-x_{n}\sqrt{is+|\xi^{\prime}|^2}}-e^{-x_{n}|\xi^{\prime}|}\right)\psi\\
&=\frac{-i\xi_k(|\xi'|+\alpha)\left(\sum_{h=1}^{n-1}i\xi_{h}\tilde{a}_{h}+(\alpha\sqrt{is+|\xi^{\prime}|^2}+|\xi'|^2) \tilde{a}_{n}\right)}{|\xi^{\prime}|\left(|\xi'|+\sqrt{is+|\xi'|^2}+\alpha\right)}L\\
&\quad-\frac{i\xi_k(|\xi'|+\alpha)\left(|\xi'|+\sqrt{is+|\xi'|^2}\right)\left(e^{-x_{n}\sqrt{is+|\xi^{\prime}|^2}}-e^{-x_{n}|\xi^{\prime}|}\right)\tilde{a}_n}{|\xi'|\left(|\xi'|+\sqrt{is+|\xi'|^2}+\alpha\right)}.
\end{align*}
We obtain that if $x_n\leq 1$,
\begin{equation}\label{R1.est}
|R_1|\lesssim |\tilde{\bm{a}}|+|\tilde{\bm{a}}||\xi'|.
\end{equation}
By dominated convergence theorem, we have that $R_1$ converge to $0$ in $L^2$. Hence, we have $\pd_n\tilde{u}_k-\alpha \tilde{u}_k\to \tilde{a}_k$ in $L^2$. Similarly, we can prove $\tilde{u}_{n}\to \tilde{a}_n$ in $L^2$. 
\end{proof}

\section{Estimate of Poisson kernel}\label{Sec 4}

In this section, we give upper bound estimates of the Poisson kernel derived in Lemma \ref{lem_p0} and Lemma \ref{lem_pa}, and extend its usage to non-smooth boundary value $\bm{a}$.
\begin{lem}\label{lem_estP}
Let $t>0$.
For $1\leq k\leq n$ and $1\leq j \leq n$,
\begin{align}\label{E P0}
|\pd_t^m\pd_x^lP^0_{kj}|\lesssim \frac{1}{(|x|^2+t)^{\frac{n+l+\delta_{jn}}{2}+m}}.
\end{align}
For $1\leq k\leq n$ and $1\leq j\leq n-1$,
\begin{equation} \label{E Pkj}
|\partial_t^m\partial_{x^{\prime}}^{l}\partial_{x_n}^{i} P_{kj}^{\alpha}|\lesssim \frac{1}{\left(|x|^2+t\right)^{\frac{l+n-1}{2}}t^m}\left(\delta_{i=0}+\frac{\delta_{i\geq1}}{\left(|x|^2+t\right)^{\frac{1}{2}}\left(x_n^2+t\right)^{\frac{i-1}{2}}}\right).
\end{equation}
For $1\leq k\leq n$ and $\sigma=\delta_{k<n}$, 
\begin{equation}\label{E Pkn}
|\partial_t^m\partial_{x^{\prime}}^{l}\partial_{x_n}^{i} P_{kn}^{\alpha}|\lesssim \frac{1}{\left(|x|^2+t\right)^{\frac{l+n-\sigma}{2}}\left(x_n^2+t\right)^{\frac{i+\sigma}{2}}t^{m}}.
\end{equation}
\end{lem}
\begin{rem}
When we evaluate the velocity $\bm{u}$ given by \eqref{split of u} and Lemmas \ref{lem_p0} and \ref{lem_pa},
$P_{kn}^\alpha$ with $1\leq k \leq n-1$ are more likely the source of singularity for high-order derivatives, because one factor of $(|x|^2+t)^{\frac{1}{2}}$ becomes $(x_n^2+t)^{\frac{1}{2}}$ in \eqref{E Pkn}. 
Moreover, comparing \eqref{E Pkn} with the estimate of the Golovkin tensor in  \cite{Solonnikov1964} or \cite[(2.20)]{Kang2023},
\[
|\partial_t^m\partial_{x^{\prime}}^{l}\partial_{x_n}^{i} K_{kn}|\lesssim \frac{1}{\left(|x|^2+t\right)^{\frac{l+n-\sigma}{2}}\left(x_n^2+t\right)^{\frac{i+\sigma}{2}}t^{m+\frac12}},\quad 1\leq k \leq n
,\]
the only difference between the estimates of $K_{kn}$ and $P_{kn}^{\alpha}$ is the exponent of $t$.
\end{rem}

\begin{proof}
For $P^0_{kj}(x,t)$ given by \eqref{3.10} and \eqref{3.11}, by
\begin{equation}\label{gam1}
|\partial_t^m\partial_x^l\Gamma(x,t)|\lesssim \frac{1}{(|x|^2+t)^{\frac{l+n}{2}+m}}e^{-\frac{|x|^2}{10t}},
\end{equation}
and for $2m+l+n\geq 3$,
\begin{equation}\label{33}
\Big|\partial_t^m\partial_x^l\int_t^{\infty}\Gamma(x,s)\ds\Big|\lesssim \frac{1}{(|x|^2+t)^{\frac{l+n-2}{2}+m}},
\end{equation}
we get \eqref{E P0}.

Next, for $P^\al_{kj}$, $j<n$, given in \eqref{pkj1}-\eqref{def of I}, we estimate $I_1$--$I_4$. We have
\begin{align*}
|\partial_t^m\partial_{x^{\prime}}^{l}\partial_{x_n}^{i} I_1|&\lesssim \int_{0}^{+\infty}\frac{1}{\left(|x|^2+z^2+t\right)^{\frac{l+i+n}{2}+m}}\dz\notag\\
&\lesssim \frac{1}{\left(|x|^2+t\right)^{\frac{l+i+n-1}{2}+m}}.
\end{align*}
Since 
$$I_2=\int_{0}^{+\infty}\!\!\int_{0}^{+\infty}e^{-\alpha (w+z)}A(x^{\prime},z,x_n+w+z)\dw\dz,$$ 
by Lemma \ref{lemAA} and
\begin{equation}\label{ze}
\int_{-\infty}^{+\infty}e^{-\frac{z^2}{10t}}\dz\lesssim\,t^{\frac{1}{2}},
\end{equation}
we have that for $l+n\geq 3$ 
\begin{align*}
|\partial_t^m\partial_{x^{\prime}}^{l}\partial_{x_n}^{i} I_2|&\lesssim \int_{0}^{+\infty}\!\!\int_{0}^{+\infty}\frac{1}{\left(|x|^2+t\right)^{\frac{l+n-2}{2}}\left(x_n^2+t\right)^{\frac{i+1}{2}+m}} e^{-\frac{w^2+z^2}{10t}}\dw\dz\notag\\
&\lesssim \frac{t^{\frac{1}{2}}}{\left(|x|^2+t\right)^{\frac{l+n-2}{2}}\left(x_n^2+t\right)^{\frac{i}{2}+m}}.
\end{align*}
By integration by parts in variable $y_n$ (to gain decay in $|x|$ when we apply Lemma \ref{lemAA}), we have
\begin{align*}
\partial_{x_n}^{i}I_3&=\int_{0}^{+\infty}e^{-\alpha z}\int_{\R_+^{n}}\partial_n^{i} E(x^{\prime}-y^{\prime},y_n+z)\Gamma(y^{\prime},x_{n}+z-y_n,t)\dy\dz\\
&\quad+(-1)^{i}\int_{0}^{+\infty}e^{-\alpha z}\int_{\R_+^{n}}\partial_n^{i} E(x^{\prime}-y^{\prime},y_n+z)\Gamma(y^{\prime},x_{n}+z+y_n,t)\dy\dz\\
&\quad+2\delta_{i\geq2}\Sigma_{0\leq\tau\leq\frac{i}{2}-1}\int_{0}^{+\infty}e^{-\alpha z}\int_{\R^{n-1}}\partial_n^{2\tau+1} E(x^{\prime}-y^{\prime},z)\partial_{n}^{i-2\tau-2}\Gamma(y^{\prime},x_{n}+z,t)\dy^{\prime}\dz.
\end{align*}          
Therefore, by Lemma \ref{lemAA} and \eqref{ze}, we have that for $l+i+n\geq 4$, 
\begin{align*}
|\partial_t^m\partial_{x^{\prime}}^{l}\partial_{x_n}^{i} I_3|&\lesssim \int_{0}^{+\infty}\int_{0}^{+\infty}\frac{e^{-\frac{(x_n+z-y_n)^2}{10t}}}{\left(|x|^2+z^2+t\right)^{\frac{l+i+n-2}{2}}t^{\frac{1}{2}+m}}\dy_{n}\dz\notag\\
&\quad+\int_{0}^{+\infty}\int_{0}^{+\infty}\frac{e^{-\frac{(x_n+z+y_n)^2}{10t}}}{\left(|x|^2+z^2+t\right)^{\frac{l+i+n-2}{2}}t^{\frac{1}{2}+m}}\dy_{n}\dz\notag\\
&\quad+\delta_{i\geq2}\int_{0}^{+\infty}\frac{e^{-\frac{z^2}{10t}}}{\left(|x|^2+t\right)^{\frac{l+n-1}{2}}\left(x_n^2+t\right)^{\frac{i-1}{2}}t^m}\dz\notag\\
&\lesssim \frac{1}{\left(|x|^2+t\right)^{\frac{l+i+n-3}{2}}t^m}+\frac{\delta_{i\geq2}}{\left(|x|^2+t\right)^{\frac{l+n-1}{2}}\left(x_n^2+t\right)^{\frac{i-2}{2}}t^m}.
\end{align*}
Note that in the first line above we have used  $|y_n+z|+|x_n+z-y_n|\geq x_n+2z$. For the two double integrals, we have first integrated in $y_n$ using \eqref{ze}, and then integrated in $z$ using $l+i+n-2>1$.

Similarly, we have that for $l+i+n\geq 4$,
\begin{align*}
|\partial_t^m\partial_{x^{\prime}}^{l}\partial_{x_n}^{i} I_4|&\lesssim \int_{0}^{+\infty}\int_{0}^{+\infty}\frac{e^{-\frac{(z-y_n)^2}{10t}}}{\left(|x|^2+z^2+t\right)^{\frac{l+i+n-2}{2}}t^{\frac{1}{2}+m}}\dy_{n}\dz\notag\\
&\quad+\int_{0}^{+\infty}\int_{0}^{+\infty}\frac{e^{-\frac{(z+y_n)^2}{10t}}}{\left(|x|^2+z^2+t\right)^{\frac{l+i+n-2}{2}}t^{\frac{1}{2}+m}}\dy_{n}\dz\notag\\
&\lesssim \frac{1}{\left(|x|^2+t\right)^{\frac{l+i+n-3}{2}}t^m}.
\end{align*}
Therefore, by \eqref{pkj1}, we arrive at \eqref{E Pkj}.

Now for $P^\al_{kn}$ given in \eqref{pkj2}-\eqref{def of J}, we estimate $J_1$--$J_3$. By Lemma \ref{lemAA} and \eqref{ze}, we have that for $l+n\geq 3$ 
\begin{align}
|\partial_t^m\partial_{x^{\prime}}^{l}\partial_{x_n}^{i} J_1|&\lesssim e^{-\frac{x_n^2}{10t}} \int_{0}^{+\infty}\frac{1}{\left(|x|^2+z^2+t\right)^{\frac{l+n-2}{2}}\left(x_n^2+t\right)^{\frac{1+i}{2}+m}} e^{-\frac{z^2}{10t}}\dz \notag\\
&\lesssim \frac{1}{\left(|x|^2+t\right)^{\frac{l+n-2}{2}}\left(x_n^2+t\right)^{\frac{i}{2}+m}}e^{-\frac{x_n^2}{10t}}.\label{J_1}
\end{align}
Analogous to the estimates of $I_3$, by integration by parts in variable $y_n$, we have 
\begin{align}
\nonumber 
\partial_{x_n}^{i}J_2&=\int_{0}^{+\infty}e^{-\alpha z}\int_{\R_+^{n}}\partial_n^{i+1} E(x^{\prime}-y^{\prime},y_n+z)\Gamma(y^{\prime},x_{n}+z-y_n,t)\dy\dz\\
\label{J2 derivative}
&\quad+(-1)^{i}\int_{0}^{+\infty}e^{-\alpha z}\int_{\R_+^{n}}\partial_n^{i+1} E(x^{\prime}-y^{\prime},y_n+z)\Gamma(y^{\prime},x_{n}+z+y_n,t)\dy\dz\\
&\quad+2\delta_{i\geq2}\Sigma_{0\leq\tau\leq\frac{i}{2}-1}\int_{0}^{+\infty}e^{-\alpha z}\int_{\R^{n-1}}\partial_n^{2\tau+2} E(x^{\prime}-y^{\prime},z)\partial_{n}^{i-2\tau-2}\Gamma(y^{\prime},x_{n}+z,t)\dy^{\prime}\dz.
\nonumber
\end{align}
By Lemma \ref{lemAA} and \eqref{ze}, we have that for $l+i+n\geq 3$,
\begin{align}
|\partial_t^m\partial_{x^{\prime}}^{l}\partial_{x_n}^{i} J_2|&\lesssim \int_{0}^{+\infty}\int_{0}^{+\infty}\frac{e^{-\frac{(x_n+z-y_n)^2}{10t}}}{\left(|x|^2+z^2+t\right)^{\frac{l+i+n-1}{2}}t^{\frac{1}{2}+m}}\dy_{n}\dz\notag\\
&\quad+\int_{0}^{+\infty}\int_{0}^{+\infty}\frac{e^{-\frac{(x_n+z+y_n)^2}{10t}}}{\left(|x|^2+z^2+t\right)^{\frac{l+i+n-1}{2}}t^{\frac{1}{2}+m}}\dy_{n}\dz\notag\\
&\quad+\delta_{i\geq2}\int_{0}^{+\infty}\frac{e^{-\frac{z^2}{10t}}}{\left(|x|^2+t\right)^{\frac{l+n}{2}}\left(x_n^2+z^2+t\right)^{\frac{i-1}{2}}t^m}\dz\notag\\
&\lesssim \frac{1}{\left(|x|^2+t\right)^{\frac{l+i+n-2}{2}}t^m}+\frac{\delta_{i\geq2}}{\left(|x|^2+t\right)^{\frac{l+n}{2}}\left(x_n^2+t\right)^{\frac{i-2}{2}}t^m}.\label{J_2}
\end{align}
By integration by parts in variable $y_n$, we have
\begin{align*}
J_3&=\int_{0}^{+\infty}e^{-\alpha z}\int_{\R_+^{n}}\partial_n E(x^{\prime}-y^{\prime},x_n+y_n+z)\Gamma(y^{\prime},z-y_n,t)\dy\dz\\
&\quad+\int_{0}^{+\infty}e^{-\alpha z}\int_{\R_+^{n}}\partial_n E(x^{\prime}-y^{\prime},x_{n}+y_n+z)\Gamma(y^{\prime},z+y_n,t)\dy\dz\\
&\quad+2\int_{0}^{+\infty}e^{-\alpha z}\int_{\R^{n-1}} E(x^{\prime}-y^{\prime},x_n+z)\Gamma(y^{\prime},z,t)\dy^{\prime}\dz.
\end{align*}
By Lemma \ref{lemAA} and \eqref{ze}, we have  that for $l+i+n\geq 3$,
\begin{align}
|\partial_t^m\partial_{x^{\prime}}^{l}\partial_{x_n}^{i} J_3|&\lesssim \int_{0}^{+\infty}\int_{0}^{+\infty}\frac{e^{-\frac{(z-y_n)^2}{10t}}}{\left(|x|^2+z^2+t\right)^{\frac{l+i+n-1}{2}}t^{\frac{1}{2}+m}}\dy_{n}\dz\notag\\
&\quad+\int_{0}^{+\infty}\int_{0}^{+\infty}\frac{e^{-\frac{(z+y_n)^2}{10t}}}{\left(|x|^2+z^2+t\right)^{\frac{l+i+n-1}{2}}t^{\frac{1}{2}+m}}\dy_{n}\dz\notag\\
&\quad+\int_{0}^{+\infty}\frac{e^{-\frac{z^2}{10t}}}{\left(|x|^2+t\right)^{\frac{l+i+n-2}{2}}t^{\frac{1}{2}+m}}\dz\notag\\
&\lesssim \frac{1}{\left(|x|^2+t\right)^{\frac{l+i+n-2}{2}}t^m}.\label{J_3}
\end{align}
Therefore, by \eqref{pkj2}, we have that for $1\leq k\leq n-1$,
\begin{equation}
|\partial_t^m\partial_{x^{\prime}}^{l}\partial_{x_n}^{i} P_{kn}^{\alpha}|\lesssim \frac{1}{\left(|x|^2+t\right)^{\frac{l+n-1}{2}}\left(x_n^2+t\right)^{\frac{i+1}{2}}t^m},
\end{equation}
and for $k=n$,
\begin{equation}
|\partial_t^m\partial_{x^{\prime}}^{l}\partial_{x_n}^{i} P_{nn}^{\alpha}|\lesssim \frac{1}{\left(|x|^2+t\right)^{\frac{l+n}{2}}\left(x_n^2+t\right)^{\frac{i}{2}}t^m}.
\end{equation}
These show \eqref{E Pkn}.
\end{proof}

\begin{lem}\label{lem_Lp P0kn}
Let $1\leq k\leq n$ and $1<q<\infty$. For all
$a(x',t) \in C^\infty_c(\R^{n-1}\times \R)$, we have
\begin{equation}\label{44}
\|(P^0_{kn}*a)(x',x_n,t)\|_{L^q_{x',t}}\leq C\|a\|_{L^q_{x',t}},
\end{equation}
where $C$ is independent of $x_n>0$. Hence the convolution $P^0_{kn}*a$ has a unique extension to all $a(x',t)\in L^q(\R^{n-1}\times \R)$ that satisfies \eqref{44}.
\end{lem}
\begin{proof}
Recall $P^0_{kn}(x,t)$ is given by \eqref{3.11}. It suffices to show \eqref{44} for $a \in C^\infty_c$.
It is easy to check that
\begin{equation}\label{555}
\|\pd_n\Gamma(x',x_n,t)\|_{L^1_{x',t}}+ \|\pd_n E(x',x_n)\|_{L^1_{x'}}\leq C,
\end{equation}
where the constant $C$ is independent of $x_n$. As for $f(x,t)=\mathcal{H}(t)\pd^3_n\int_t^{\infty}\Gamma(x,s)ds$, it's not hard to see $f(x,t)=\frac{1}{x_n^{n+1}}f(\frac{x'}{x_n},1,\frac{t}{x_n^2})$. Thus, 
\begin{equation*}
\|f(x',x_n,t)\|_{L^1_{x',t}}=\|f(x',1,t)\|_{L^1_{x',t}}\leq C.
\end{equation*}
Hence, by Young's inequality, we obtain \eqref{44} for $k=n$. 

For $1\leq k\leq n-1$, by \cite[page 29, Theorem 1]{Stein70}, we can prove 
\begin{equation}\label{ECZ}
\|\pd_kE(x) *_{x'}a\|_{L^q_{x',t}}\leq C \|a\|_{L^q_{x',t}}.
\end{equation}  
As for $g(x,t)=\mathcal{H}(t)\pd_k\pd^2_n\int_t^{\infty}\Gamma(x,s)\ds$ with $1\leq k \leq n-1$, by \eqref{CZ2}, we have
\begin{align*}
\mathcal{F}_{x',t}(g)&=\frac{1}{2(2\pi)^{\frac{n}{2}}} \frac{i\xi_{k}}{is}\left(\sqrt{is+|\xi^{\prime}|^2}e^{-x_{n}\sqrt{is+|\xi^{\prime}|^2}}-|\xi'|e^{-x_{n}|\xi^\prime|}\right)\\
&=\frac{1}{2(2\pi)^{\frac{n}{2}}} \frac{i\xi_k}{\sqrt{is+|\xi'|^2}+|\xi'|}\cdot \frac{\sqrt{is+|\xi'|^2}e^{-x_n\sqrt{is+|\xi'|^2}}-|\xi'|e^{-x_n|\xi'|}}{\sqrt{is+|\xi'|^2}-|\xi'|}
\end{align*}
which is bounded uniformly in $x_n>0$. Thus, the inequality
\begin{equation}\label{gCZ}
\|g*a\|_{L^q_{x',t}}\leq C \|a\|_{L^q_{x',t}}
\end{equation}
holds for $q=2$. By the estimate \eqref{33} we can verify the conditions for \cite[page 19, Theorem 3]{Stein93} are fulfilled, which shows that \eqref{gCZ} holds for $q\leq 2$. Using duality argument we can get that \eqref{gCZ} holds for $q\geq 2$. Combining \eqref{gCZ} and \eqref{ECZ}, we arrive at \eqref{44}.
\end{proof}

\begin{lem}\label{Lq-estimate-kernel}
Suppose $\bm{a}(x',t)\in L^q_c(\R^{n-1}\times \R)$, $1<q<\infty$, and $\bm{u}(x,t)$ is constructed through the Poisson kernel by \eqref{u_expression}.
Denote $\Omega=\hbox{supp}(\bm{a})$. For any compact set $K\subset\overline{\R^n_+}\times\R$, we have 
\begin{align}\label{cor_est1}
\|\bm{u}\|_{L^q(K)}\leq C(K,\Omega)\|\bm{a}\|_{L^q(\Omega)}.
\end{align}
Moreover, $\bm{u}$ is spatially differentiable to any order when $x_n>0$.

\end{lem}

\begin{proof}
Recall $u_k=P_{kj}*a_j$ with $P_{kj}=P^0_{kj}+\alpha P^\alpha_{kj}$. By Lemma \ref{lem_Lp P0kn}, we have that 
\[\|P^0_{kn}*a\|_{L^q(K)}\leq C\|a\|_{L^q(\Omega)}.\]
By Lemmas \ref{lem_p0} and \ref{lem_pa}, $P^0_{kj}$, $P^\alpha_{kj}$, $j<n$ and $P^\alpha_{kn}$ are functions on $\overline{\R^{n}_+}\times \R$ that vanish for $t\le 0$.
By Lemma \ref{lem_estP}, we have for $1\leq k\leq n$, $1\leq j\leq n-1$, $t>0$ and $\sigma=\delta_{k<n}$,
\[|P^0_{kj}|\lesssim \frac{1}{(|x'|^2+t)^{\frac{n}{2}}}, \ \ \ |P^\alpha_{kj}|\lesssim \frac{1}{(|x'|^2+t)^{\frac{n-1}{2}}},\ \ \ |P^\alpha_{kn}|\lesssim \frac{1}{(|x'|^2+t)^{\frac{n-\sigma}{2}}t^{\frac{\sigma}{2}}}.\]
Therefore, $P_{kj}^{0},\,P_{kj}^{\alpha},\,P_{kn}^{\alpha}\in L^{1}_{\loc}(\R^{n-1}\times\R)$, uniformly in $x_n \ge 0$.

By Young's convolution inequality, \eqref{cor_est1} is proved. Also by Lemma \ref{lem_estP}, $\bm{u}$ is spatially differentiable to any order when $x_n>0$, and we have
\begin{equation}\label{xnn}
	|\pd_{x}^l \bm{u}|\leq C (1+x_n^{-l-n-1}).\qedhere
\end{equation}
\end{proof}

\smallskip

\begin{proof}[Proof of Theorem \ref{thm_main}]
First suppose $\bm{a}\in C_c^{\infty}$. Similar to \eqref{xnn}, $\bm{u}$, $p$ are smooth functions for both variables $x$ and $t$. Similar to the proof of Lemma \ref{lem 5-2}, $\bm{u}$, $p$ ($\nabla p$ if $n=2$) vanish sufficient fast near infinity.
Since $\tilde{\bm{u}},\tilde{p}$ are given in \eqref{FT of u} and \eqref{FT of pressure}, it is not hard to verify they satisfy Stokes equations in Fourier space. Hence $\bm{u}$, $p$ satisfy Stokes equations \eqref{Stokes_Eq} in physical space. By Lemma \ref{lem-boundary-convergence} we know $\bm{u}$ satisfies Navier boundary condition \eqref{Navier-BC}. Hence, $\bm{u}$ is a weak solution satisfying \eqref{cor_weak form3}. Then using Lemma \ref{Lq-estimate-kernel} and a density argument, we get the general case for $\bm{a}\in L^q_c$. 
\end{proof}

\section{Blow-up of the second derivative in $L^q$}\label{Sec 5}
In this section, we will prove Theorem \ref{thm_main2}. In the following we assume $1<q<\infty$, since $L^q$ blow-up in $Q_{\frac{1}{2}}^+$ implies $L^\infty$ blow up in $Q_{\frac{1}{2}}^+$.  Using a similar idea as that in \cite{Chang2023}, we set the boundary value 
\begin{equation*}
\bm{a}(x',t)=g(x^{\prime},t)e_n, \ \ \ e_n=(0,...,0,1),
\end{equation*}
for scalar function $g(x^{\prime},t)=g^{\cS}(x')g^{\cT}(t)$, where $g^{\cS}(x')$ is a cutoff function
\begin{equation}\label{gs}
g^{\cS}(x^{\prime})\in C_c^{\infty}(B^{\prime}_1(x'_0)), \quad x'_0=(-4,0,\cdots,0),
\end{equation}
and
\begin{equation}\label{gt}
\mathrm{supp}\  g^{\cT}(t) \subset (\frac{3}{4},\frac{7}{8}),\quad g^{\cT}(t)\in L^{\infty}(\R)\setminus \dot{B}^{\frac{1}{2}-\frac{1}{2q}}_{q,q}(\R).
\end{equation}
There is an example of function $g^{\cT}$ satisfying \eqref{gt}
in \cite[Appendix C]{Chang2023}. We will take $g^{\cS}\geq 0$ for $n\geq 3$, and $g^{\cS}=\pd_1g^{\cS}_1$ with $g^{\cS}_1\geq 0$ for $n=2$. 

By Theorem \ref{thm_main}, a velocity field satisfying Stokes equation \eqref{Stokes_Eq} and Navier BC \eqref{Navier-BC} with the given boundary data $\bm{a}$ is
\begin{align}\label{exp of u}
\begin{split}
u_k(x,t)&=\int_0^t\int_{\R^{n-1}}\big(P^0_{kn}+\alpha P^{\alpha}_{kn}\big)(\xi^{\prime},x_n,s)g^{\cS}(x'-\xi')g^{\cT}(t-s)\dxi' \ds\\
&\equiv u_k^0+\alpha u_k^{\alpha},
\end{split}
\end{align}
for $1\leq k \leq n$, where $P_{kn}^{0}$ and $P_{kn}^{\alpha}$ are defined in \eqref{3.11} and \eqref{pkj2}. By \eqref{pkj2}, we set
\begin{equation}
u_k^\alpha=-4 \delta_{k<n}u_{k,1}^{\alpha} -2 u_{k,2}^{\alpha}+2 u_{k,3}^{\alpha},
\end{equation}
where for $1\leq \ell \leq3$,
\begin{equation}\label{ukl}
u_{k,\ell}^{\alpha}=\int_0^t\int_{\R^{n-1}}\partial_k\partial_n J_\ell(\xi^{\prime},x_n,s)a_n(x^{\prime}-\xi',t-s)\dxi' \ds.
\end{equation}
We give some identities which will be often used in this section. By the fact
\begin{equation}\label{EGamma}
\partial_n^2E=-\Delta_{x'}E, \quad \partial_n^2\Gamma=-\Delta_{x'}\Gamma+\partial_t\Gamma,
\end{equation}
and \eqref{J2 derivative} for $\partial_n^2 J_2$, we obtain that for $J_1,J_2, J_3$ defined in \eqref{def of J},
\begin{align}\label{differential transfer}
\partial_n^2J_1=-\Delta_{x'}J_1+\partial_tJ_1, \ \ \ 
\partial_n^2 J_2=-\Delta_{x'}J_2-2\Delta_{x'}J_1, \ \ \ 
\partial_n^2 J_3=-\Delta_{x'}J_3.
\end{align}
The first identity in \eqref{differential transfer} will not be used.

\begin{lem}\label{lem 5-1}
For $\bm{u}$ defined in \eqref{exp of u}, we have $|\bm{u}(x,t)|+|\nabla \bm{u}(x,t)|\lesssim 1 $, when $|x|\leq 10$ and $0<t\leq 2$.
\end{lem}
\begin{proof}
We first consider the case $n\geq 3$. 

First, we estimate $u_{k}^0$, $1\leq k \leq n$. Notice that by  \eqref{EGamma},
\begin{equation}\label{6.5}
\pd_n^2\int_t^{\infty}\Gamma(x,s)\ds=-\Gamma(x,t)-\Delta_{x'}\int_t^{\infty}\Gamma(x,s)\ds,
\end{equation} 
and by \eqref{33},
\begin{align}\label{61-1}
\int_0^2\int_{B'_{20}}\left( |\Gamma|+\left|\pd_x\int_t^{\infty}\Gamma(x,s)\ds\right|\right)\dx'\dt\lesssim 1.
\end{align}
By \eqref{555}, \eqref{6.5}--\eqref{61-1} and integration by parts of the horizontal derivatives $\partial_{x'}^{m}$, $m\geq 1$, it is easy to verify
\begin{equation}\label{5.12}
|\bm{u}^0|+|\nabla \bm{u}^0|\lesssim 1.
\end{equation}

Next we estimate $u^\alpha_k$, $1\leq k \leq n$. Using \eqref{differential transfer} and integration by parts of the horizontal derivatives $\partial_{x'}^{m}$, $m\geq 1$, we get 
\begin{equation}
	\label{eq511jj}
|\bm{u}^\alpha|+|\partial_{x'}\bm{u}^{\alpha}|\lec \int_0^2\int_{B'_{20}}\left( |J_1|+|\pd_n J_1|+|J_2|+|\pd_n J_2|+|J_3|+|\pd_n J_3|\right)\dxi'\ds.
\end{equation}
By estimates \eqref{J_1}, \eqref{J_2} and \eqref{J_3}, the RHS of the above inequality is bounded, i.e.,
\begin{equation}\label{5.13}
|\bm{u}^\alpha|+|\partial_{x'}\bm{u}^{\alpha}|\lec 1.
\end{equation}
Analogously for $1\leq k \leq n$, we have
\begin{equation}\label{5.10}
|\pd_n u^{\alpha}_{k,2}|+|\pd_nu^{\alpha}_{k,3}|\lesssim 1,
\end{equation}
where $u_{k,\ell}^{\alpha}$ is defined in \eqref{ukl}. Now we look at normal derivative of $ u_{k,1}^{\alpha}$, $1\leq k\leq n-1$,
\begin{equation*}
\pd_nu^{\alpha}_{k,1}=-\int_0^t\int_{\R^{n-1}}\partial_n^2 J_1(\xi^{\prime},x_n,s)\partial_k a_n(x^{\prime}-\xi',t-s)\dxi' \ds.
\end{equation*}
We first integrate by parts in variable $x_n$ to get
\begin{align}\label{J_1 decompose}
\partial_{n}J_1 =\int_{0}^{+\infty}e^{-\alpha z}\int_{\R^{n-1}}E(x^{\prime}-y^{\prime},z)\pd_z\Gamma(y^{\prime},x_{n}+z,t)\dy^{\prime}\dz= -K_1+\alpha J_1-K_2, 
\end{align}
where 
\[K_1=\int_{\R^{n-1}}E(y',0)\Gamma (x'-y',x_n,t)\dy',\]
\[K_2=\int_{0}^{+\infty}e^{-\alpha z}\int_{\R^{n-1}}\pd_n E(x^{\prime}-y^{\prime},z)\Gamma(y^{\prime},x_{n}+z,t)\dy^{\prime}\dz.\]
We can obtain that by \eqref{J_1},
\begin{equation}
|\partial_{n} J_1|\lec \frac{1}{|x|^{n-2}\cdot t^{\frac{1}{2}}},
\end{equation}
by Lemma \ref{lemAA}, 
\begin{equation}\label{K_1}
| \partial_{n} \partial_{x'}^m K_1|=\frac{1}{\sqrt{4\pi t}}|\partial_{x'}^m A(x',0,0,t)|\cdot|\partial_{n}e^{-\frac{x_n^2}{4t}}|\lesssim \frac{x_n}{|x|^{n+m-2}\cdot t^{\frac{3}{2}}}e^{-\frac{x_n^2}{4t}},\quad m\geq 0,
\end{equation}
and by Young's convolution inequality, \eqref{555} and \eqref{ze}
\begin{align}
\nonumber
\|\pd_{n} K_2\|_{L^1_{x'}(\R^{n-1})}&\lec \int_{0}^{+\infty}e^{-\alpha z}\|\pd_n E(x^{\prime},z)\|_{L^1_{x'}(\R^{n-1})}\|\pd_n\Gamma(y^{\prime},x_{n}+z,t)\|_{L^1_{y'}(\R^{n-1})}\dz\\
&\lec \int_{0}^{+\infty}1\cdot1\cdot t^{-1} e^{-\frac{z^2}{8t}}\dz\lec t^{-\frac{1}{2}}.
\label{K_2}
\end{align}
Therefore, we have
\begin{equation}\label{518kk}
\int_0^2\int_{B'_{20}}|\pd_nK_1|+|\pd_nJ_1|+|\pd_{n} K_2| \dxi'\ds\lec 1.
\end{equation}
Hence, we get that for $1\leq k \leq n-1$,
\begin{equation}\label{5.14}
|\pd_n u_{k,1}^{\alpha}(x,t)|\lesssim 1.
\end{equation}
By \eqref{5.10}  and \eqref{5.14}, we have
\begin{equation}\label{5.21}
|\pd_n\bm{u}^{\alpha}|\lesssim1.
\end{equation} 
Combining \eqref{5.12}, \eqref{5.13}, and \eqref{5.21}, we have proved Lemma \ref{lem 5-1} for $n \ge 3$.

For $n=2$, extra care is needed because we need $l\geq 1$ in Lemma \ref{lemAA}. We can use the following estimate
\begin{align}\label{6.7}
\begin{split}
& \ \ \ |\int_{B'_{20}}\partial_{1}J_1(\xi',x_n)g^{\cS}(x'-\xi')\dxi'|=|\int_{B'_{20}}\partial_{1}J(\xi',x_n)(g^{\cS}(x'-\xi')-g^{\cS}(x'))\dxi'|\\
&\lesssim\int_{B'_{20}}|\partial_{1}J(\xi',x_n)|\cdot |\xi'|\cdot\|Dg^{\cS}\|_{L^{\infty}}\dxi' \lesssim \int_{B'_{20}}\frac{1}{|\xi'|^{n-2}}\dxi'\lesssim1,
\end{split}
\end{align}
on the estimate of $J_{k}$ with $1\leq k\leq 3$ and $K_1$, and 
replacing \eqref{eq511jj} by
\[
|\bm{u}^\alpha|+|\partial_{x'}\bm{u}^{\alpha}|\lec \int_0^2\int_{B'_{20}}\sum_{k=1}^{3}\left( |\pd_1 J_k|+|\pd_1\pd_n J_k|\right)\dxi'\ds,
\]
and similarly for \eqref{518kk}.
The other parts are the same as $n\geq 3$.
\end{proof}
\begin{lem}\label{lem 5-2}
For $\bm{u}$ defined in \eqref{exp of u}, we have that for $|x|> 10$ and $0<t\leq 2$,
\begin{equation}\label{666}
|\bm{u}(x,t)|+|\nabla \bm{u}(x,t)|\lesssim \frac{1}{|x|^{n-1}}.
\end{equation}
\end{lem}
\begin{proof}
By \eqref{E P0}, \eqref{E Pkn}, $\frac{1}{|x-\xi'|}\lesssim \frac{1}{|x|}$ when $\xi'\in B_5'(x'_0)$, and $\int_0^2 t^{-1/2}\dt \lec 1$, we have
\begin{equation}\label{5.111}
|\bm{u}|+|\nabla \bm{u}^0(x,t)|+|\partial_{x'}\bm{u}^{\alpha}|+|\partial_{n} u_n^{\alpha}|\lec \frac{1}{|x|^{n-1}}.
\end{equation} 
Now we come to the estimate of $\pd_n u_k^\alpha$ with $1\leq k\leq n-1$. By \eqref{J_2} and \eqref{J_3}, we have
\begin{equation}
|\pd_{n} u_{k,2}^{\alpha}|+ |\pd_{n} u_{k,3}^{\alpha}| \lec \frac{1}{|x|^{n-1}}.
\end{equation}
By Lemma \ref{lemAA}, we have
\begin{equation}\label{K_20}
|\pd_k\pd_nK_2|\lec \frac{1}{|x|^{n}\cdot t^{\frac{1}{2}}}.
\end{equation}
By the decomposition of $\pd_n J_1$ in \eqref{J_1 decompose} and \eqref{J_1}, \eqref{K_1}, \eqref{K_20}, we have that for $1\leq k \leq n-1$
\begin{equation}\label{5.11}
|\pd_n u_{k,1}^{\alpha}|\lesssim \frac{1}{|x|^{n-1}}.
\end{equation}
Thus, we have proved this lemma.
\end{proof}
Thus, by Lemma \ref{lem 5-1} and \ref{lem 5-2}, we get the finite global energy \eqref{global energy} of $\bm{u}$ for $n\geq 3$. For $n=2$, by moving the derivative $\pd_1$ in $g^{\cS}(x')=\pd_1g^{\cS}_1$ to  $P_{kn}$ in \eqref{exp of u}, we are able to get a stronger decay estimate, for $|x|>10$ and $0<t\leq 2$,
\begin{equation}\label{eq5.26}
|\bm{u}(x,t)|+|\nabla \bm{u}(x,t)|\lesssim \frac{1}{|x|^{n}}.
\end{equation}
Thus, it has finite energy \eqref{global energy}. We leave the details to the interested readers. Now, we show the blow-up of second derivative of the solution in $L^{q}(Q_{\frac{1}{2}}^+)$, $1<q<\infty$.
\begin{lem}\label{lem 5-3}
	For $\bm{u}$ defined in \eqref{exp of u} and $1<q<\infty$, we have 
	\begin{equation}
		\|\pd^2_nu_1\|_{L^{q}(Q_{\frac{1}{2}}^+)}= \infty.
	\end{equation}
\end{lem}
\begin{proof}
	For $(x,t)\in Q^+_1$, we can obtain that by \eqref{E P0}, 
	\begin{equation}\label{5.28}
		|\pd_n^2 \bm{u}^{0}|\lec1,
	\end{equation}
	and by estimate \eqref{J_2} and \eqref{J_3}, both with $i=3$, 
	\begin{equation}
		|\pd_n^2 u_{1,2}^{\alpha}|+|\pd_n^2 u_{1,3}^{\alpha}|\lec1.
	\end{equation}
	Next we work on $\pd^2_nu^{\alpha}_{1,1}$
	\[\pd^2_nu^{\alpha}_{1,1}=\int_0^t\int_{B'_1(x_0')}\partial_1\partial_n^3 J_1(x^{\prime}-\xi^{\prime},x_n,s)a_n(\xi',t-s)\dxi' \ds.\]
	We still use the decomposition \eqref{J_1 decompose} to get 
	\[\pd_n^3J_1=-\pd_n^2K_1+\alpha \pd_n^2J_1-\pd_n^2K_2.\]
	For $\pd_n^2J_1$, 
	by \eqref{5.14}, we know that for $(x,t)\in Q_1^+$,
	\begin{equation}\label{5.16}
		\Big|\int_0^t\int_{B'_1(x_0')}\pd_1\pd^2_nJ_1(x'-\xi',x_n,s)a_n(\xi',t-s)\dxi'\ds\Big|=|\pd_nu^{\alpha}_{1,1}|\lesssim1.
	\end{equation}
	Moreover, for $\pd_n^2K_2$, using the fact that $-2\pd_nE(x',x_n)\to \delta(x')$ as $x_n\to 0^+$, we have
	\begin{equation*}
		\lim_{z\rightarrow0^+}\int_{\R^{n-1}}\pd_n E(x^{\prime}-y^{\prime},z)\Gamma(y^{\prime},x_{n}+z,t)\dy^{\prime}=-\frac{1}{2}\Gamma(x,t).
	\end{equation*}
	Thus, by integration by parts of variable $z$, we have
	\begin{equation*}
		\pd_{n} K_2=\frac{1}{2}\Gamma(x,t)+\alpha K_2-K_3,
	\end{equation*}
	where
	\begin{equation*}
		K_3=\int_{0}^{+\infty}e^{-\alpha z}\int_{\R^{n-1}}\pd_n^2 E(x^{\prime}-y^{\prime},z)\Gamma(y^{\prime},x_{n}+z,t)\dy^{\prime}\dz.
	\end{equation*}
	By Lemma \ref{lemAA} and \eqref{ze}, we have
	\begin{equation}\label{K_3}
		|\pd_{1}\pd_n K_3|\lec \frac{1}{|x|^{n+1}\cdot t^{\frac{1}{2}}}.
	\end{equation}
	Therefore, by \eqref{gam1}, \eqref{K_20}, \eqref{K_3}, and $|x'-\xi'| \ge 3$ for $\xi'\in \operatorname{supp} g^{\cS}$, we are able to obtain
	\begin{equation}\label{5.17}
		\Big|\int_0^t\int_{B'_1(x_0')}\pd_1\pd^2_nK_2(x'-\xi',x_n,s)a_n(\xi',t-s)\dxi'\ds\Big|\lesssim 1.
	\end{equation}
	
	Next, we denote the leading blow-up term from $\pd_n^2K_1$:
	\begin{equation}\label{uB}
		H(x,t)=\int_0^t\int_{B'_1(x_0')}\pd_1\pd_n^2 K_1(x^{\prime}-\xi^{\prime},x_n,s)a_n(\xi',t-s)\dxi' \ds,
	\end{equation}
	where 
	\[\pd_1\pd^2_nK_1(x,t)=\int_{\R^{n-1}}E(y',0)\pd_1\pd_n^2\Gamma (x'-y',x_n,t)\dy'
		=\mathrm{p.v.}\!
	\int_{\R^{n-1}}\pd_1E(y',0)\pd_n^2\Gamma (x'-y',x_n,t)\dy',
	\]
	using $\lim_{\ep \to 0} \int _{y'\in\R^{n-1},|y'|=\ep} E(y',0)\pd_n^2\Gamma (x'-y',x_n,t)e_1 \cdot \frac {y'}{\ep}\,dS_{y'}=0$.
		By Lemma \ref{lem210}, we can decompose $H$ into 
	\[H=H_1+H_2,\]
	where
	\[H_1=\int_0^t\int_{B'_1(x_0')} \pd_n^2\Gamma_1(x_n,s)\pd_1E(x'-\xi',0)a_n(\xi',t-s)\dxi' \ds,\]
	$\Gamma_1$ is the 1D heat kernel,
	and
	\[H_2=-\frac{1}{n(4\pi)^{\frac{n}{2}}|B_1|}\int_0^t\int_{B'_1(x_0')} \partial_n^2\big(e^{-\frac{x_n^2}{4s}}\big)Err(x'-\xi',s)\frac{1}{|x'-\xi'|^n}a_n(\xi',t-s)\dxi' \ds.\]
	For $(x,t)\in Q_1^+ $,
	\begin{align*}
		|H_2|&\lesssim \Big|\int_0^t\partial_n^2\big(e^{-\frac{x_n^2}{4s}}\big)\ds\Big|\lesssim\int_0^t \frac{1}{s}e^{-\frac{x_n^2}{8s}}\ds\lesssim \ln(2+\frac{1}{x_n^2}).
	\end{align*}
	Hence, we have that for $1<q<\infty$,
	\begin{equation}\label{u_b2}
		\|H_2\|_{L^{q}(Q_{1}^+)}\lesssim1.
	\end{equation}

	Next, we factor $H_1$ as 
	\begin{equation}\label{ub1}
		H_1=\psi(x') \theta(x_n,t),
	\end{equation}
	where
	\begin{equation}\label{key blow-up term}
		\theta(x_n,t)=\int_0^t \pd_n^2\Gamma_1(x_n,s)g^{\cT}(t-s)\ds,
	\end{equation}
	for $n\geq 3$, 
	\begin{equation*}
		\psi(x')=\int_{\R^{n-1}}\pd_1E(x'-\xi',0)g^{\cS}(\xi')\dxi'=-\frac{1}{n|B_1|}\int_{\R^{n-1}}\frac{x_1-\xi_1}{|x'-\xi'|^n}\cdot g^{\cS}(\xi')\dxi',
	\end{equation*}
	and for $n=2$,
	\[\psi(x')=-\frac{1}{n|B_1|}\int_{-5}^{-3}\frac{x_1-\xi_1}{|x_1-\xi_1|^2}\cdot \pd_1g^{\cS}_1(\xi')\dxi'=\frac{1}{n|B_1|}\int_{-5}^{-3}\frac{1}{(x_1-\xi_1)^2}\cdot g^{\cS}_1(\xi')\dxi'.\]
	Actually for $x'\in B'_1$, we have that $2\leq x_1-\xi_1\leq 6$. Hence for $n\geq 3$ we have $C_1\leq \psi(x')\leq C_2<0$, and for $n=2$ we have $0<C_1\leq \psi(x')\leq C_2$. Hence we have $\|\psi\|_{L^q(B'_{\frac{1}{2}})}\neq 0$.

	Since $g^{\cT}(t)$ is supported in $(\frac{3}{4},\frac{7}{8})$, $\theta(x_n,t)$ is defined for $(x_n,t) \in \R\times \R$, and vanishes for $t<3/4$. 
		We have for $1<t<\infty$,
	\[|\theta(x_n,t)|\lesssim t^{-3/2}e^{-\frac{ x_n^2}{10t}} \cdot\|g^{\cT}\|_{L^{\infty}(\R)},\]
	and for $0<t\leq1$,
	\begin{align*}
		|\theta(x_n,t)|\lesssim
		\int_0^{t}\frac{1}{s^{\frac{3}{2}}}e^{-\frac{x_n^2}{8s}}\ds\cdot\|g^{\cT}\|_{L^{\infty}(\R)}\lesssim x_n^{-1}\cdot\|g^{\cT}\|_{L^{\infty}(\R)}.
	\end{align*}
	This implies that for $1<q<\infty$, using \eqref{ze},
	\begin{align}\label{ub2}
		\begin{split}
			\int_1^{\infty}\int_0^{\infty}|\theta(x_n,t)|^q\dx_n\dt\lesssim \|g^{\cT}\|_{L^{\infty}(\R)},\\
			\int_{0}^{1}\int_{\frac{1}{2}}^{\infty}|\theta(x_n,t)|^q\dx_n\dt\lesssim \|g^{\cT}\|_{L^{\infty}(\R)}.
		\end{split}
	\end{align}
	Note that both inequalities above are not true when $q=1$. By Lemma \ref{Besov est}, we have 
	\begin{align}\label{ub3}
		\infty=\|g^{\cT}\|_{\dot{B}_{q,q}^{\frac{1}{2}-\frac{1}{2q}}}\lesssim \|\theta(x_n,t)\|_{L^{q}(\R_+\times \R)}.
	\end{align}
	With the aid of \eqref{ub1}, \eqref{ub2}, \eqref{ub3}, and the fact that $\theta(x_n,t)=0$ when $t<\frac{3}{4}$, we conclude that 
	\begin{align}\label{key blow-up term2}
		\begin{split}
			\|\theta(x_n,t)\|_{L^{q}((0,\frac{1}{2})\times (\frac{3}{4},1))}&=\|\theta(x_n,t)\|_{L^{q}((0,\frac{1}{2})\times (0,1))}\\
			&\geq \|\theta(x_n,t)\|_{L^{q}(\R_+\times \R)}-c\|g^{\cT}\|_{L^{\infty}(\R)}=\infty.
		\end{split}
	\end{align}
	Hence
	\begin{align*}
		\|H_1\|_{L^{q}(Q_{\frac{1}{2}}^+)}&=\|\theta(x_n,t)\|_{L^{q}((0,\frac{1}{2})\times (\frac{3}{4},1))} \|\psi\|_{L^q(B'_{\frac{1}{2}})}=\infty.
	\end{align*}
	By the above unboundedness and \eqref{5.28}-\eqref{5.16}, \eqref{5.17}, \eqref{u_b2}, we obtain that
	\begin{equation}
		\|\partial_n^2u_1\|_{L^{q}(Q_{\frac{1}{2}}^+)}=\infty.\qedhere
	\end{equation}
\end{proof}

Theorem \ref{thm_main2} now follows from Lemmas \ref{lem 5-1}-\ref{lem 5-3}. 

\begin{rem}\label{rem5.4}
	When we further require the boundary data $\bm{a}$ to be continuous, we cannot use $g^{\cT}(t)\in L^{\infty}(\R)\setminus \dot{B}^{\frac{1}{2}-\frac{1}{2q}}_{q,q}(\R)$ satisfying \eqref{gt}.
	Instead, we can choose $g^{\cT}\in C_c(\frac{1}{4},1]\cap C^{1}(0,1)$ and $g^{\cT}=(1-t)^{\beta}$ when $t\in(\frac{1}{2},1)$, for $\beta\in (0,\frac{1}{2}]$. Then we can prove $\bm{u}$ has unbounded second derivative near boundary at time $t=1$, similar to \cite{Kang2023}.

\end{rem}

\section{Shear flow example}\label{Sec 6}
In this section,  we prove Theorem \ref{thm_main3}. We follow Seregin-\v Sver\'ak \cite{Seregin2010}
and look for shear flow solutions of the Stokes system \eqref{Stokes_Eq}--\eqref{Navier-BC} in $\R^n_+ \times (0,2)$ with Navier boundary value $\bm{a}=0$ in the form
\begin{equation}\label{uu}
\bm{u}(x,t)=(v(x_n,t),0,\ldots,0), \ \ p(x,t)=-g'(t)x_1,
\end{equation}
where  $g(t)\in C_{c}^{\infty}(\R_+)$.
It's easy to see that the convection term $\bm{u}\cdot\nabla\bm{u}$ is zero, so it is also a solution of Navier--Stokes equations. Denote $y=x_n>0$, and we can reformulate the original equations as
\begin{equation}\label{1D-Stokes}
\partial_t v(y,t)-v_{yy}(y,t)=g'(t),\ \ \ v_y(0,t)-\alpha v(0,t)= 0, \ \ \ v(y,0)=0.
\end{equation}
For the heat equation with Robin BC \cite{Keller1981}, the solution is given by
\begin{equation}\label{v1}
v(y,t)=\int_{0}^t g'(t-\tau)\int_0^{\infty}G(y,\xi,\tau)\dxi \dtau,
\end{equation}
using its Green function
\begin{align*}
G(y,\xi,t)&=\Gamma(y-\xi,t)-\Gamma(y+\xi,t)-2\pd_{\xi}\int_0^{\infty}e^{-\alpha z}\Gamma(y+\xi+z,t)\dz.
\end{align*}
Here $\Gamma(y,t)$ is 1D heat kernel. Direct calculation gives
\begin{align}\label{Ge}
\int_{0}^{\infty}G(y,\xi,t)\dxi
&=2\int_{0}^{y}\Gamma(\xi,t)\dxi 
+2\int_{0}^{\infty} e^{-\alpha z} \Gamma(y+z,t)\dz,\\
\partial_{y}\int_{0}^{\infty}G(y,\xi,t)\dxi
&=2\alpha\int_{0}^{\infty} e^{-\alpha z} \Gamma(y+z,t)\dz.\notag
\end{align}
Since
\begin{equation*}
v_y-\alpha v=-2\alpha\int_{0}^{t} g'(t-\tau) \int_{0}^{y}\Gamma(\xi,\tau)\dxi\dtau \rightarrow 0,\ \  \textrm{as}\ y\rightarrow0,
\end{equation*}
the boundary condition is satisfied.
It is easy to check $v(y,t)$ in \eqref{v1} is a smooth solution (up to the boundary $y=0$) of heat equation \eqref{1D-Stokes}.

By \eqref{Ge}, we get
\begin{equation*}
\lim_{t\rightarrow 0} \int_{0}^{\infty}G(y,\xi,t)\dxi=1,
\end{equation*}
and
\begin{equation*}
\partial_{t}\int_{0}^{\infty}G(y,\xi,t)\dxi=-2\alpha\Gamma(y,t)+2\alpha^2\int_0^{\infty}e^{-\alpha z}\Gamma(y+z,t)\dz.
\end{equation*}
Using integration by parts of the derivative $\partial_{t}$, we can obtain an equivalent formula of \eqref{v1},
\begin{align}\label{shear_flow1}
v(y,t)&=g(t)+\int_{0}^t g(t-\tau)\partial_{\tau}\int_0^{\infty}G(y,\xi,\tau)\dxi \dtau\notag\\
&=g(t)-2\alpha\int_0^t g(t-\tau)\Gamma(y,\tau)\dtau+2\alpha^2\int_0^t g(t-\tau)\int_0^\infty e^{-\alpha z}\Gamma(y+z,\tau)\dz\dtau.
\end{align}
Notice that 
\[|\Gamma(y,t)|+\left|\int_0^\infty e^{-\alpha z}\Gamma(y+z,t)\dz\right|\lesssim \frac{1}{\sqrt{t}},\]
so by Young's convolution inequality,
\begin{align*}
\|v(y,\cdot)\|_{L^q(0,T)}\lesssim (1+T^{\frac{1}{2}})\,\|g(t)\|_{L^q(0,T)},
\end{align*}
for $1\leq q \leq \infty$. Since $C_c^{\infty}$ is dense in $L^{q}$ for $1\leq q<\infty$, by the density argument it is easy to check that for any $g\in L^q_c(\R_+)$ with $1\leq   q<\infty$, $v(y,t)$ in \eqref{shear_flow1} satisfies the weak form of 
\eqref{1D-Stokes}, and hence $\bm{u}$ given by \eqref{uu} satisfies the weak form of
Stokes systems \eqref{cor_weak form3} with Navier boundary value $\bm{a}=0$.

\begin{lem}
Fix $1<q<\infty$.
Let $g(t)=g^{\cT}(t)$ satisfy \eqref{gt}, $v(y,t)$ be given by \eqref{shear_flow1}, and the solution  $\bm{u}$ be given by \eqref{uu}. We have 
\begin{align}\label{shear_flow2}
\|\bm{u}\|_{L^{\infty}(Q_1^+)}+\|\nabla \bm{u}\|_{L^{\infty}(Q_1^+)}<\infty, \ \ \ \|\nabla^2 \bm{u}\|_{L^{q}(Q^+_{\frac{1}{2}})}=\infty.
\end{align}
\end{lem}
\begin{proof}
Notice that, by taking derivative of \eqref{shear_flow1},
\begin{align*}
\pd_y v(y,t)&=-2\alpha\int_0^tg(t-\tau) \pd_y\Gamma(y,\tau)\dtau-2\alpha^2\int_0^t g(t-\tau)\Gamma(y,\tau)\dtau\\
& \ \ \ +2\alpha^3\int_0^t g(t-\tau)\int_0^\infty e^{-\alpha z}\Gamma(y+z,\tau)\dz\dtau,
\end{align*}
and
\begin{equation}\label{6.88}
v_{yy}(y,t)=-2\alpha\int_0^tg(t-\tau) \pd^2_y\Gamma(y,\tau)\dtau+\alpha v_y(y,t).
\end{equation}
For $0<t\leq1$, $y\in\R_+$
\begin{equation}\label{6.99}
|v(y,t)|+|v_y(y,t)|\lesssim 1+\int_{0}^{t}\tau^{-\frac{1}{2}}\dtau+\int_{0}^{t}\frac{y}{\tau^{\frac{3}{2}}}e^{-\frac{y^2}{4\tau}}\dtau \lesssim1.
\end{equation}
Denote the main term of $v_{yy}$ in \eqref{6.88} same as \eqref{key blow-up term}
\begin{equation*}
\theta(y,t)=\int_0^tg(t-\tau) \pd^2_y\Gamma(y,\tau)\dtau.
\end{equation*}
By the same argument in the proof of Lemma \ref{lem 5-3}, we get the same estimate as \eqref{key blow-up term2},
\begin{equation}\label{6.10}
\|\theta(y,t)\|_{L^{q}((0,\frac{1}{2})\times(\frac{3}{4},1))}=\infty.
\end{equation}
By \eqref{6.88}, \eqref{6.99} and \eqref{6.10}, we reach \eqref{shear_flow2}.
\end{proof}
The above lemma proves Theorem \ref{thm_main3} for the $1<q<\infty$, and the $L^\infty$ blow-up in $Q_{\frac{1}{2}}^+$ follows from $L^q$ blow-up in $Q_{\frac{1}{2}}^+$.

\begin{rem}\label{rem62}
This shear flow example can also be used to prove the blow-up of the gradient of velocity near boundary under zero-BC, i.e., \cite[Theorem 1.1]{Chang2023}. We can use the same ansatz \eqref{uu}, \eqref{v1}, with the Green function replaced by
\begin{align*}
	G(y,\xi,t)&=\Gamma(y-\xi,t)-\Gamma(y+\xi,t),
\end{align*}
and setting $g=g^{\cT}$ satisfying \eqref{gt}.\hfill$\square$
\end{rem}

\begin{rem}\label{rem63}
If we set pressure $p=-g^{\cT}x_1$ in \eqref{uu} (instead of $p=-(g^{\cT})'x_1$), we can prove the following estimates, 
\begin{align*}
\|\bm{u}\|_{L^{\infty}(Q_1^+)}+\|\nabla^3 \bm{u}\|_{L^{\infty}(Q_1^+)}+\|p\|_{L^{\infty}(Q_1^+)}<\infty, \\
\|\nabla^4 \bm{u}\|_{L^{q}(Q^+_{\frac{1}{2}})}=\infty.
\end{align*}
In other words, velocity blows up at fourth derivative.
This is a supplementary result of estimate \eqref{1.7}.
\hfill$\square$
\end{rem}

\section*{Acknowledgments}
We warmly thank K.~Kang for fruitful discussions on his paper \cite{Chang2023} with T.~Chang. We also thank Rulin Kuan for reference \cite{Schwartz1966} for Remark \ref{rem22}.
Hui Chen was supported in part by National Natural Science Foundation of China under grant [12101556] and Zhejiang Provincial Natural Science Foundation of China under  grant [LY24A010015]. 
The research of both Liang and Tsai was partially supported by Natural Sciences and Engineering Research Council of Canada (NSERC) under grant RGPIN-2023-04534.

\bibliography{NavierBC2}
\bibliographystyle{abbrv}

\end{document}